\documentclass[a4paper,11pt]{article}
\usepackage{amsmath}
\usepackage{amsfonts}
\usepackage{eucal}
\usepackage{amsthm}

\newcommand{\bigset}[2]{\bigl\{#1\bigm|#2\bigr\}}

\newcommand{\bigabs}[1]{\bigl| #1 \bigr|}

\newcommand{\jap}[1]{\langle #1 \rangle}

\def\a{\alpha}

\def\d{\delta}

\def\n{\nu}

\def\s{\sigma}

\def\x{\xi}
\def\y{\eta}

\def\C{\Gamma}

\newcommand{\be}{\begin{equation}}
\newcommand{\ee}{\end{equation}}

\def\la{\langle}
\def\ra{\rangle}

\def\T{{\bf T}}

\def\re{\mathbb{R}}
\def\co{\mathbb{C}}

\def\na{\mathbb{N}}
\def\pa{\partial}

\renewcommand{\Re}{\text{{\rm Re}\;}}
\renewcommand{\Im}{\text{{\rm Im}\;}}

\newtheorem{thm}{Theorem}[section]
\newtheorem{lem}[thm]{Lemma}
\newtheorem{prop}[thm]{Proposition}

\theoremstyle{definition}

\newtheorem{ass}{Assumption}

\newtheorem{rem}[thm]{Remark}

\newcommand{\WF}{W\!F_a}

\numberwithin{equation}{section}

\title{Analytic Wave Front Set for Solutions to Schr\"odinger Equations II -- Long Range Perturbations}
\author{Andr\'e Martinez${}^1$, Shu Nakamura${}^2$, Vania Sordoni${}^1$}

\begin{document}

\maketitle 
\addtocounter{footnote}{1}
\footnotetext{Universit\`a di Bologna, Dipartimento di
Matematica, Piazza di Porta San Donato 5, 40127 Bologna,
Italy. Partly supported by Universit\`a di Bologna, Funds
for Selected Research Topics and Founds for Agreements with
Foreign Universities}
\addtocounter{footnote}{1}
\footnotetext{Graduate School of Mathematical Science, University of
Tokyo, 3-8-1
Komaba, Meguro-ku, Tokyo, Japan 153-8914.}

\begin{abstract} 
This paper is a continuation of \cite{MNS2}, where short range perturbations of the flat Euclidian metric where considered. Here, we generalize the results of \cite{MNS2} to long-range perturbations (in particular, we can allow potentials growing like $\langle x\rangle^{2-\varepsilon}$ at infinity). More precisely, we construct a modified quantum free evolution $G_0(-s, hD_z)$ acting on Sj\"ostrand's spaces, and we characterize the analytic wave front set of the solution $e^{-itH}u_0$ of the Schr\"odinger equation, in terms of the semiclassical exponential decay of $G_0(-th^{-1}, hD_z)\T u_0$, where $\T$ stands for the Bargmann-transform. 
The result is valid for $t<0$ near the forward non trapping points, and for $t>0$ near the backward non trapping points. It is an extension  of \cite{Na4} to the analytic framework.
\end{abstract}


\section{Introduction} \label{sec-intro}
We consider the analytic singularities of the solutions $u(t)=e^{-itH}u_0$ to  a variable coefficients 
Schr\"odinger equation, where the Schr\"odigner operator $H$ is time-independent 
and of long-range type perturbation  (that is, sub-quadratic) of the Laplacian $H_0$ on $\re^n$.

\bigskip
Such a problem has been the source of an abundant literature in the last decades, and we refer to \cite{MNS1, MNS2} for a long (though probably not exhaustive) list of references. Let us only mention the most recent works \cite{Do1, Do2, HaWu, It, KaTa, MNS1, MNS2, MRZ, Na2, Na3, Na4, RoZu1, RoZu2, RoZu3, Wu}.

\bigskip
In \cite{MNS2}, we proved that, in the short range case, the forward (resp. backward) non-trapping microlocal singularities propagate for $t<0$ (resp. $t>0$) accordingly with those of the free evolution $e^{-itH_0}u_0$, except for a shift due to the possible perturbation of the metric (no shift appears if the perturbation is of the first order). Actually, this shift is expressed by the underlying classical wave operators for the pair $(H,H_0)$, that is, the map $S_{\pm}$ defined by
$$
S_{\pm}(x,\xi ):= \lim_{t\rightarrow \pm\infty }\exp tH_{p_0}\circ \exp tH_{q_0} (x,\xi),
$$
where $q_0$ and $p_0$ are the principal symbols of $H$ and $H_0$, respectively. More precisely, denoting by $FNT$ (resp. $BNT$) the forward (resp. backward) non-trapping set, we proved the two identities (see \cite{MNS2} Theorem 2.1):
\begin{eqnarray*}
&& WF_a(e^{-itH}u_0)\cap FNT = S_+^{-1}(WF_a(e^{-itH_0}u_0)) \,\, \mbox{ for all } t<0;\\
&& WF_a(e^{-itH}u_0)\cap BNT = S_-^{-1}(WF_a(e^{-itH_0}u_0)) \,\, \mbox{ for all } t>0.
\end{eqnarray*}

\bigskip
However, in the long range case, the previous operators $S_{\pm}$ do not exist anymore, and, as well as in the corresponding quantum case, one has to modify the free evolution near infinity in order to define wave operators. 

\bigskip
Here, we follow the general idea of \cite{Na4}, that consisted in replacing the free quantum evolution by an operator of the form $e^{iW(-t, D_x)}$, where $W=W(t,\xi)$ is a solution to $\partial_t W=p(\partial_\xi W, \xi)$ for large $|\xi|$, and $p$ is the {\it total} symbol of $H$. 

\bigskip
But in our case, we have the additional difficulty that we must remain in the analytic category, and thus, avoid the use of cut-off functions. 

\bigskip
In order to solve this problem, we prefer to work from the very beginning in weighted Sj\"ostrand's spaces, since this allows us to put ourselves in a semiclassical setting, and to limit the construction of the modified free evolution to the set $\{ |\xi|>\delta_0\}$, with $\delta_0>0$ arbitrarily small (and actually, we could even have limited it to a compact subset of $\re^n\backslash 0$). In this way, we obtain an analytic ($h$-dependent) function $W(s, \xi;h)$, solution of
$$
\partial_sW = h^2 p(\partial_\xi W, h^{-1}\xi),
$$
where $h>0$ is the additional semiclassical parameter, and our result can be written in terms of decaying properties, as $h\rightarrow 0_+$, of the quantity $e^{iW(-th^{-1}, hD_z)/h}\T u_0$, where $\T$ is the usual  Bargmann transform: $L^2(\re^n)\rightarrow H_{\Phi_0}^{loc}$, and the operator $e^{iW(-th^{-1}, hD_z)/h}$ acts on the weigthed Sj\"ostrand space $H_{\Phi_0}^{loc}$ (see Theorem \ref{mainth} for a precise statement).

\bigskip
Let us also observe that we recover one of the difficulties of \cite{MNS2}, concerning the fact that the size of the region in which the solution must be considered increases like $\la s\ra =\la th^{-1}\ra$. In \cite{MNS2}, this appeared just after the conjugation by $e^{-itH_0}$, and constituted the main problem in order to apply Sj\"ostrand's theory (see \cite{MNS2} Lemma 3.1). Here, this difficulty appears repeatedly when we want to make changes of good contours in the integrals. While this can be done automatically in the microlocal setting of \cite{Sj}, here we need to justify it each time we do it, because the size of these contours increases, too, like $\la s\ra$.

\bigskip
The paper is organized as follows:

\bigskip
In the next section, we introduce the notations and state our main result. In Section \ref{secPrelim}, we prove  several estimates on the Hamilton flow of the total semiclassical symbol of $h^2H$. In Section \ref{sec-modevol}, we construct both the classical modified free evolution and the quantum modified free evolution, and we prove that the modified evolution acts correctly on convenient Sj\"ostrand spaces. In Section \ref{sec-conj}, we conjugate the evolution equation by the modified free quantum evolution, and we study the structure of the resulting equation. Then, the proof of the main theorem is completed in Section \ref{sec-proofmainth}. The appendices contain the justification of the various changes of contours of integration (Appendix \ref{appA}) and a technical result concerning the derivations on non-local Sj\"ostrand's spaces (Appendix \ref{appB}).

\vskip 0.2cm

\section{Notations and result}\label{sec-result}

We consider the  Schr\"odinger equation associated with the operator,
\[
H= \frac12 \sum_{j,k=1}^n D_j a_{j,k}(x)D_k 
+  \frac12\sum_{j=1}^n (a_j(x)D_j+D_ja_j(x)) +a_0(x)
\]
on $\mathcal{H}=L^2(\re^n)$, where $D_j = -i\pa_{x_j}$. We suppose 
the coefficients $\{a_\a(x)\}$ satisfy to the following 
assumptions. For $\n>0$ we denote
\[
\C_\n=\bigset{z\in\co^n}{|\Im z|<\n \jap{\Re z}}.
\]
\begin{ass} 
\label{assA}
For each $\a$, $a_\a(x)\in C^\infty(\re^n)$ is real-valued and  can be extended to a holomorphic 
function on $\C_\n$ with some $\n>0$. Moreover, for $x\in \re^n$, the matrix  $(a_{j,k}(x))_{1\leq j, k\leq n}$
is symmetric and positive definite, and  there exists $\s\in (0,1]$ 
such that, 
\begin{align*}
&\bigabs{a_{j,k}(x)-\d_{j,k}} 
\leq C_0 \jap{x}^{-\s}, \quad j,k=1,\dots,n, \\
&\bigabs{a_{j}(x)} 
\leq C_0\jap{x}^{1-\s} , \quad\qquad j=1,\dots,n,\\
&\bigabs{ a_0(x)} 
\leq C_0 \jap{x}^{2-\s},
\end{align*}
for $x\in\C_\n$ and with some constant $C_0>0$. 
\end{ass}

(Here, we have used the notation $\la x\ra := (1+|x|^2)^{1/2}$.) In particular, $H$ is essentially selfadjoint on $C_0^\infty (\re^n)$, and, denoting by the same letter $H$ its unique selfadjoint extension on $L^2(\re^n)$, we can consider
its quantum  evolution group $e^{-itH}$.

\bigskip
In order to describe the analytic wave-front set of $u$, we use the setting of \cite{Sj} and introduce the Bargmann-FBI transform $\T$ defined by,
$$
\T u(z,h)=\int e^{-(z-y)^2/2h}u(y)dy,
$$
where $z\in\co^n$ and $h>0$ is a small extra-parameter. Then, $\T u$ belongs to the Sj\"ostrand space $H_{\Phi_0}^{loc}$ with $\Phi_0(z):= |\Im z|^2/2$, that is, it is a holomorphic function of $z$, and, for any compact set $K\subset\co^n$ and any $\varepsilon >0$, there exits $C=C(k,\varepsilon)$ such that $|\T u(z,h)|\leq Ce^{(\Phi_0(z)+\varepsilon)/h}$, uniformly for $z\in K$ and $h>0$ small enough.

\bigskip

We recall from \cite{Sj} that a point $(x,\x)\in T^*\re^n\backslash 0$ is not in $\WF (u)$ if and only if there exists some $\delta >0$ such that $\T u ={\cal O}(e^{(\Phi_0(z)-\delta)/h})$ uniformly for $z$ close enough to $x-i\x$ and $h>0$ small enough. By using Cauchy-formula and the continuity of $\Phi_0$, it is easy to see that this is also equivalent to the existence of some $\delta'>0$ such that  $\Vert e^{-\Phi_0/h}\T u\Vert_{L^2(\Omega)} ={\cal O}(e^{-\delta' /h})$ for some complex neighborhood $\Omega$ of $x-i\xi$. In that case, we will just write: $\T u\sim 0$ in $H_{\Phi_0, x-i\xi}$, where $H_{\Phi_0, z}$ is the space of  germs of $H_{\Phi_0}$-functions near a complex point $z$ (see \cite{Sj} and the appendix of \cite{MNS2}).

\bigskip
We denote by $p(x,\xi):=\frac12 \sum_{j,k=1}^n a_{j,k}(x)\x_j\x_k $ the  principal symbol of $H$, and by $H_0:= -\frac12\Delta$ the free Laplace operator. For any $(x,\x)\in \re^{2n}$, we also denote by $(y(t,x,\x), \y(t,x,\x))={\rm exp}tH_p(x,\x)$ the Hamilton flow of $p$, and we say that a point $(x_0,\x_0)\in T^*\re^n\backslash 0$ is forward non-trapping when $|y(t,x_0,\x_0)|\rightarrow \infty$ as $t\rightarrow +\infty$. In this case, it is well-known that $\eta (t,x_0,\xi_0)$ admits a limit $\x_+(x_0,\x_0)\in\re^n\backslash 0$ as $t\rightarrow +\infty$. However, in contrast with the short-range case, the quantity $y(t,x_0,\x_0)-t\eta(t,x_0,\x_0)$ may not have a limit. 

\bigskip
In order to overcome this inconvenience, one has to modify the free evolution near infinity. For $h>0$ small enough, we set,
$$
q(x,\xi ;h):= \frac12 \sum_{j,k=1}^n a_{j,k}(x)\xi_j\xi_k 
+  h\sum_{j=1}^n a_j(x)\xi_j+h^2a_0(x).
$$
Then, given some $\delta_0>0$ arbitrarily small, and following \cite{Na4}, for $s\geq 0$, $|\xi|>\delta_0$, and $h>0$ small enough, in Section \ref{secmodevol} we construct a function $W(s,\xi ;h)$, solution of,
\be
\label{eikonal1}
\frac{\partial W}{\partial s} -q(\partial_\xi W, \xi ;h) =0,
\ee
and such that, denoting by $(\tilde x (s, z,\zeta;h ), \tilde\xi (s, z,\zeta ;h)):=\exp sH_q(x,\xi )$ the  Hamilton flow of $q$, then, for any forward non-trapping point $(x_0,\xi_0)$ with $|\xi_+(x_0,\xi_0)|>\delta_0$, the quantity,
\be
\label{modevol}
\tilde x (h^{-1}, x_0,\xi_0 ;h) - \partial_\xi W (h^{-1}, \tilde\xi (h^{-1}, x_0,\xi_0 ;h);h)+\partial_\xi W (0, \xi_+( x_0,\xi_0) ;h)
\ee
admits a limit $x_+(x_0,\xi_0)$  as $h$ tends to $0_+$. 

\bigskip 
For $z\in \co^n\cap\{| \Im z |>\delta_0\}$, we set,
\begin{eqnarray*}
&& Z(s,z):=z+\partial_\xi W (0,-\Im z) - \partial_\xi W (s,-\Im z),
\end{eqnarray*}
and we denote by $W(s, hD_z)$ a quantization of $W(s,\zeta)$ on $H_{\Phi_0, z}$ as in \cite{Sj} (see also \cite{MNS2}, Appendix). Then, for any $s\geq 0$, in Section \ref{secmodevolq} we construct an invertible analytic  Fourier Integral Operator,
$$
G_0(s) \, :\, H_{\Phi_0,Z(s,z)} \rightarrow H_{\Phi_0, z}
$$
such that,
$$
ih\frac{\partial G_0}{\partial s} +(\partial_s W)(s,hD_z) G_0(s) \sim 0\quad ; \quad G_0(0)\sim I.
$$
For more transparency in the notations, we will write $e^{i \widetilde W(s, hD_z)/h}$ for the operator $G_0(s)$, where $\widetilde W(s, \zeta):= W(s,\zeta )-W(0,\zeta )$.

\bigskip
Then, our main result is,
\begin{thm}\sl \label{mainth}
Suppose Assumption~A, and suppose $(x_0,\x_0)$ is forward non-trapping with $|\xi_+(x_0,\xi_0)|>\delta_0$. Then, for any $t<0$ and  any $u_0\in L^2(\re^n)$,  one has the equivalence,
$$
(x_0,\x_0)\notin \WF(e^{-itH}u_0)\, \iff \,  e^{i \widetilde W(-th^{-1}, hD_z)/h}\T u_0\sim 0 \mbox { in } H_{\Phi_0, z_+(x_0,\xi_0)},
$$
where $z_+(x_0,\xi_0):=x_+(x_0,\xi_0)-i\xi_+(x_0,\xi_0)$.
\end{thm}
\begin{rem} Since $WF_a(u)$ is conical with respect to $\xi$ and $\xi_+(x_0, \lambda\xi_0)=\lambda\xi_+(x_0,\xi_0)$ for all $\lambda >0$, the condition $|\xi_+(x_0,\xi_0)|>\delta_0$ is not restrictive.
\end{rem}
\begin{rem} Actually, equation (\ref{eikonal1}) needs not be satisfied by $W$, and the result remains valid with any $W$ such that (\ref{modevol}) admits a limit, and $\frac{\partial W}{\partial s} -q(\partial_\xi W, \xi ;h) = {\cal O}(\la s\ra^{-1-\sigma})$ uniformly for $s={\cal O}(h^{-1})$, $h\rightarrow 0_+$.
\end{rem}
\begin{rem} In the short-range case, one can actually take $W(s,\xi) = s\xi^2/2$, so that $e^{i W(-th^{-1}, hD_z)/h}$ just becomes $e^{-itD_z^2/2}$, and the function $e^{i W(-th^{-1}, hD_z)/h}\T u_0$ coincides with $\T (e^{-itH_0}u_0)$. Thus, in that case, one recovers the result of \cite{MNS2}.
\end{rem}
\begin{rem} Of course, there is a similar result for  $(x_0,\x_0)$  backward non-trapping  and $t>0$.
\end{rem}

\section{Preliminaries}\label{secPrelim}
Replacing $u_0$ by $e^{itH}u_0$ and changing $t$ to $-t$, we see that the result can be reformulated by writing that, for any $t>0$, one has the equivalence,
$$
(x_0,\x_0)\notin \WF(u_0)\, \iff \,  e^{i\widetilde W(th^{-1}, hD_z)/h}\T (e^{-itH}u_0)\sim 0 \mbox { in } H_{\Phi_0, z_+(x_0,\xi_0)}.
$$
We set $v(t):=\T (e^{-itH}u_0)$. Then, by a standard result of Sj\"ostrand's theory (see \cite{Sj} Proposition 7.4 and \cite{MNS2} Section 4),  we see that $v(t)$ is solution of,
\begin{equation}
\label{eqevol}
i\frac{\partial v}{\partial t} \sim  \tilde H v(t) \, \mbox{ in } H_{\Phi_0}^{loc},
\end{equation}
where $\tilde H$ is the pseudodifferential operator on $H_{\Phi_0}^{loc}$, defined by,
\be
\label{defHtilde1}
\tilde H:=\frac12 \sum_{j,k=1}^n D_{z_j} {\rm Op}_R(\tilde a_{j,k})D_{z_k} 
+\frac12\sum_{j=1}^n ({\rm Op}_R(\tilde a_j)D_{z_j}+D_{z_j}{\rm Op}_R(\tilde a_j)) +{\rm Op}_R(\tilde a_0).
\ee
Here, we have set $\tilde a_{j,k} (z,\zeta ) := a_{j,k}(z+i\zeta )$, $\tilde a_{j} (z,\zeta ) := a_{j}(z+i\zeta )$, and, for any function $a(z.\zeta )$ holomorphic  near some point $(z_0, -\Im z_0)$, we have denoted by  ${\rm Op}_R(a)$ its quantization on $H_{\Phi_0,z_0}$  given by,
\be
\label{quantizPDO}
{\rm Op}_R(a) v(z;h):= \frac1{(2\pi h)^n}\int_{\gamma_R (z)}e^{i(z-y)\zeta /h}a\big( \frac{y+z}2 ,\zeta\big) v(y)dyd\zeta,
\ee
where $\gamma_R (z)$ is the $2n$-complex contour,
\be
\label{contourPDO}
\gamma_R (z) \, :\, \zeta = -\Im z +iR(\overline{z-y}) \,\, ;\,\, |z-y| < R^{-1/2},
\ee
with $R >0$ constant, $R$ sufficiently (and arbitrarily) large. 

\bigskip
As in \cite{MNS2}, we change the time scale by setting $s:=t/h$, and we multiply equation (\ref{eqevol}) by $h^2$. We obtain (with the notation $\tilde v(s):=v(hs)$),
\begin{equation}
\label{eqevol2}
ih\frac{\partial \tilde v}{\partial s} \sim \widetilde Q \tilde v(s) \, \mbox{ in } H_{\Phi_0}^{loc},
\end{equation}
with $\widetilde Q:=h^2\tilde H = \widetilde Q_0 +h\widetilde Q_1 +h^2\widetilde Q_2$, where the symbol of  $\widetilde Q_j$ is of the form $\widetilde q_j +{\cal O}(h)$ (locally in $(z,\zeta)$), with,
\begin{eqnarray}
&& \widetilde q_0(z,\zeta ;h) = q_0(z+i\zeta, \zeta) :=\frac12 \sum_{j,k=1}^n a_{j,k}(z+i\zeta) \zeta_j\zeta_k ;\nonumber\\
\label{defqj}
&& \widetilde q_1(z,\zeta ;h) = q_1(z+i\zeta, \zeta) := \sum_{j=1}^n a_{j}(z+i\zeta) \zeta_j ;\\
&& \widetilde q_2(z,\zeta ;h) =  q_2(z+i\zeta, \zeta) := a_{0}(z+i\zeta).\nonumber
\end{eqnarray}

\bigskip
In order to construct the function $W(s,\xi)$ and the Fourier integral operator $G_0(s)\sim e^{i \widetilde W(s,hD_z)/h}$, we need some estimates on the Hamilton flow of $q:=q_0 + hq_1 +h^2q_2$.

\begin{lem} \sl
\label{estflow}
Set $(\tilde x (s, x,\xi;h ), \tilde\xi (s, x,\xi ;h)):=\exp sH_q(x,\xi )$. Then, for any forward non-trapping point $(x_0,\xi_0)$ and any $T>0$, $\tilde\xi (T/h, x_0,\xi_0 ;h)$ tends to $\xi_+(x_0,\xi_0)$ (independent of $T$) as $h\rightarrow 0_+$.
Moreover,  there exists a constant $C=C(T)>0$ such that, for $(x,\xi)\in \co^{2n}$ close enough to $(x_0,\xi_0)$, $s\in [0, T/h]$, and $h>0$ small enough, one has,
\begin{eqnarray*}
&& |\tilde x (s, x,\xi;h )| \geq \frac{s}C  -C ;\\
&& |\tilde\xi (s, x,\xi ;h) -\xi_+( x, \xi)|\leq C\la s\ra^{-\sigma};\\
&& | \tilde x (s, x,\xi;h )| \leq Cs+C.
\end{eqnarray*}
\end{lem}
\begin{proof} This proof is rather standard, and we just sketch it (see, e.g., \cite{Na4} for more details). At first, we observe that Assumtion A implies that the flow exists for all $s\geq 0$ if the starting point is close enough to the real, and (omitting the dependence with respect to $x,\xi,h$ in the notations), we compute,
\be
\label{d2s2}
\frac{d^2}{ds^2}|\tilde x (s )|^2=2|\tilde\xi (s)|^2+U(\tilde x(s),\tilde\xi (s)),
\ee
with,
$$
U(x,\xi )={\cal O}(\la x\ra^{-\sigma}|\xi|^2+h\la x\ra^{1-\sigma}|\xi| +h^2\la x\ra^{2-\sigma})
$$
uniformly. Now, by the conservation of energy and Assumption A, we see that $|\tilde \xi (s)|+|\tilde \xi (s)|^{-1}$ remains uniformly bounded, while $\tilde x (s) ={\cal O}(\la s\ra)$ uniformly. Therefore, for $s\in [0,T/h]$, we deduce from (\ref{d2s2}) and Assumption A,
\be
\label{x(s)convex}
\frac{d^2}{ds^2}|\tilde x (s )|^2\geq C^{-1}-Ch^{\sigma}-C\la \tilde x(s)\ra^{-\sigma},
\ee
for some constant $C>0$ and $(x,\xi)\in\co^{2n}$ close enough to $(x_0,\xi_0)$. Since $(x_0,\xi_0)$ is forward non trapping, there necessarily exists $s_0>0$ such that $\la \tilde x(s_0)\ra^{\sigma}> 3C^2$ and $\partial_s |\tilde x(s_0)| >0$ (it is true at $(x,\xi)=(x_0,\xi_0)$, and thus also in a complex neighborhood of this point by continuity of the flow). Then, by (\ref{x(s)convex}), and for $h$ small enough, we deduce that $|\tilde x(s)|^2$ is a convex function of $s$ in $[s_0, T/h]$, and that $|\tilde x(s)|^2\geq (s-s_0)^2/2C$ for all $s\in[s_0,T/h]$. 

\bigskip
The same arguments apply to the flow $(y(s),\eta (s)):=\exp tH_{q_0}(x,\xi)$ of the principal symbol $q_0$ (independent of $h$) of $H$, and then we can compare $(\tilde x(s), \tilde\xi (s))$ with $(y(s),\eta (s))$. A direct computation, as in the proof of \cite{Na4} Proposition 2.9, leads to,
\begin{eqnarray*}
&& |\partial_s(\tilde x -y ) |\leq C(|\tilde\xi -\eta |+\la s\ra^{-1-\sigma}|\tilde x-y| +h\la s\ra^{1-\sigma});\\
&& |\partial_s(\tilde \xi -\eta)| \leq C(\la s\ra^{-1-\sigma}|\tilde\xi -\eta |+\la s\ra^{-2-\sigma}|\tilde x -y|+ h\la s\ra^{-\sigma}+h^2\la s\ra^{1-\sigma}).
\end{eqnarray*}
For $s\in[0,T/h]$, we set  $g(s) =C\int_s^{+\infty} \la s'\ra^{-1-\sigma}ds'$ and  $(\alpha, \beta):=e^{g}(|\tilde x-y|, |\tilde\xi -\eta|)$. For almost all $s\in [0,T/h]$, we obtain,
\begin{eqnarray}
\label{eqalpha}
&& \partial_s\alpha \leq C(\beta+h\la s\ra^{1-\sigma});\\
&& \partial_s\beta \leq C(\la s\ra^{-2-\sigma}\alpha+ h\la s\ra^{-\sigma}).
\end{eqnarray}
Setting 
\begin{eqnarray*}
Y&:=& \int_0^s\la s'\ra^{-1-\sigma/2}(\partial_s\alpha)(s')ds' + \beta\\
&{}=&\la s\ra^{-1-\sigma/2}\alpha+\beta+(1+\sigma/2)\int_0^s\la s'\ra^{-2-\sigma/2}\frac{s'}{\la s'\ra}\alpha (s')ds',
\end{eqnarray*}
we find,
$$
\partial_sY\leq C\la s\ra^{-1-\sigma/2}Y +2Ch\la s\ra^{-\sigma} \,\, \mbox{ a.e.},
$$
and thus, since $Y(0)=0$,
$$
Y(s)={\cal O}(h\la s\ra^{1-\sigma}).
$$
In particular, this gives $\tilde\xi (s) =\eta(s)+{\cal O}(h\la s\ra^{1-\sigma})$, and thus $\tilde\xi (T/h) =\eta(T/h)+{\cal O}(h^{\sigma})\rightarrow \xi_+(x,\xi)$ as $h\rightarrow 0_+$. Moreover, since $|\eta (s)-\xi_+(x,\xi)|={\cal O}(\la s\ra^{-\sigma})$, we also have $|\tilde\xi (s)-\xi_+(x,\xi)|={\cal O}(h\la s\ra^{1-\sigma}+\la s\ra^{-\sigma})={\cal O}(\la s\ra^{-\sigma})$. 
\end{proof}
\begin{rem} We also deduce from (\ref{eqalpha}) and the estimate on $\beta$ that $|\tilde x(s) -y(s)|={\cal O}(h\la s\ra^{2-\sigma})={\cal O}(h^\sigma\la s\ra)$.
\end{rem}

\section{Construction of the modified free evolution}\label{sec-modevol}

\subsection{The  modified free classical  evolution}
\label{secmodevol}

We first show,
\begin{lem} For any $\delta>0$, there exists  $R_\delta>0$, such that, for all $\xi\in\re^n$ with $|\xi|\geq\delta$, the point $X_\delta(\xi):= (R_\delta \xi /|\xi|,\xi)$ is forward non trapping. Moreover, 
for any $s\geq 0$ and
  $h>0$ small enough, the application,
$$
J_{s,\delta} \, :\, \xi \mapsto \tilde\xi (s, X_\delta(\xi) ;h)
$$
is a diffeomorphism from $\{|\xi|>\delta\}$ to its image, and there exists $\delta'=\delta'(\delta)\rightarrow 0_+$ as $\delta\rightarrow 0_+$ such that,
\be
\label{imJdelta}
J_{s,\delta} (|\xi|>\delta)\supset \{|\xi|> \delta'\}.
\ee
\end{lem}
\begin{proof} The existence of $R_\delta$ such that $X_\delta(\xi)$ is non trapping is very standard, and comes form the fact that the point $X_\delta(\xi)$ is in the out-going region (because $R_\delta \xi/|\xi|\cdot \xi = R_\delta|\xi|=\big| R_\delta \xi/|\xi|\big|\cdot |\xi |$), and that the norm of its position is $R_\delta >>1$ (see, e.g., \cite{Na4}). Then, arguments similar to (but simpler than)  those used in the proof of Lemma \ref{estflow} show that one has,
\begin{eqnarray*}
&&\big|d_\xi \tilde\xi (s, X_\delta(\xi) ;h)-I\big| ={\cal O}(R_\delta^{-\sigma});\\
&&\big|\tilde\xi (s, X_\delta(\xi) ;h)-\xi\big| ={\cal O}(R_\delta^{-\sigma}|\xi|),
\end{eqnarray*}
uniformly with respect to $s$, $\xi$, $R_\delta$ and $h$.
Therefore, if $R_\delta$ is large enough, we see that $J_{s,\delta}$ is a diffeomorphism from $\{|\xi| >\delta\}$ to its image, and that this one contains $\{ |\xi|>(1+CR_{\delta}^{-\sigma})\delta\}$ for some constant $C>0$ independent of $\delta$. Thus the result follows.
\end{proof}
Now, we fix $\delta_0>0$ arbitrarily small, and , for $s\geq 0$ and $|\xi|\geq \delta_0$, we set,
\begin{eqnarray*}
&& \hat x(s, \xi ):= \tilde x(s,R_\delta\xi/|\xi|, J_{s,\delta}^{-1}(\xi)),
\end{eqnarray*}
where $\delta>0$ is sufficiently small in order to have (\ref{imJdelta}) with $\delta'=\delta_0$.
Then, we define,
\be
\label{defW}
W(s,\xi) := R_\delta |\xi| +\int_0^s q(\hat x(s',\xi), \xi)ds'.
\ee
By standard Hamilton-Jacobi theory (see, e.g., \cite{ReSi, Ro, Na4}), we know that $W$ solves the equation,
\be
\label{eikonal2}
\frac{\partial W}{\partial s} =q(\partial_\xi W, \xi ;h),
\ee
and that one has,
\be
\label{lienWflot}
\partial_\xi W ( s, \xi) = \hat x (s,\xi).
\ee
(Indeed,  one easily verifies that  $\hat x (s,\xi)$ is solution of the equation $\partial_s \hat x (s,\xi)= {}^t\partial_\xi \hat x\, \partial_x q(\hat x ,\xi) +\partial_\xi q(\hat x ,\xi)$, with $\hat x (0,\xi)=\partial_\xi (R_\delta |\xi|)$, so that (\ref{lienWflot}) follows by differentiating (\ref{defW}) in $\xi$, and (\ref{eikonal2}) as well by derivating (\ref{defW}) in $s$.)

Moreover, $W$ is analytic on $\{|\xi|> \delta_0\}$, it is real if $\xi$ is real, and we have,
\begin{lem}\sl 
\label{modclassevol}
Let $(x,\xi)\in\re^{2n}$ be forward non trapping with $|\xi_+(x,\xi)|>\delta_0$. Then,  there exists $x_+(x,\xi)\in \re^n$ such that, for any $T>0$, the quantity,
$$
\tilde x(T/h,x,\xi;h) - \partial_\xi \widetilde W(T/h,\tilde\xi (T/h, x,\xi) ;h)
$$
tends to $x_+(x,\xi )$ as $h\rightarrow 0_+$. 
\end{lem}
\begin{proof}  The proof  is identical to that of \cite{Na4}, Proposition 2.12, and we omit it.
\end{proof}

\subsection{The modified  free quantum evolution}
\label{secmodevolq}

For any $z_0\in\co^n$ with $|\Im z_0| >\delta_0$,
the holomorphic function $\zeta\mapsto W(s,\zeta ;h)$ (defined in a complex neighborhood of  $-\Im z$)  can be quantized into an analytic pseudodifferential operator $W(s,hD_z)$ acting on $H_{\Phi_0, z_0}$,  by the formula,
$$
W(s,hD_z)v(z):=\frac1{(2\pi h)^n}\int_{\gamma (z)}e^{i(z-y)\zeta /h}W(s ,\zeta) v(y)dyd\zeta,
$$
where $\gamma (z)$ is the $2n$-complex contour,
$$
\gamma (z) \, :\, \zeta = -\Im z +iR(\overline{z-y}) \,\, ;\,\, |z-y| < r,
$$
with $r >0$ constant, sufficiently (and arbitrarily) small and $R\geq 1$ arbitrarilly large.

\bigskip
The purpose of this section is to construct  a Fourier integral operator,
$G_0(s)$ between some Sj\"ostrand's space $H_{\Phi_0, Z_s(z_0)}$ and  $H_{\Phi_0, z_0}$,
such that $Z_0(z_0)=z_0$, and,
$$
ih\frac{\partial G_0}{\partial s} +  (\partial_s W)(s,hD_z) G_0(s) \sim 0\quad ; \quad G_0(0)\sim I,
$$
uniformly for $s\in [0, T/h]$ ($T>0$ arbitrary). Since $-\partial_sW=-\partial_sW(s,\zeta)$ does not depend on $z$, its classical flow is easily determined as, 
$$
U_s\, :\, (z,\zeta )\mapsto (z - \partial_\zeta \widetilde W (s,\zeta ) , \zeta ),
$$
where we have set $\widetilde W(s,z):=W(s,\eta) -W(0,\eta )$. Therefore, in order to obtain $z_0$ as for the final base-point, the initial base-point  should be,
\be
\label{defZs}
Z_s(z_0):= z_0 +\partial_\zeta \widetilde W (s,-\Im z_0).
\ee
Moreover, we see that $G_0(s)$ can be taken of the form,
\be
\label{defG0}
G_0(s)v(z;h) =\frac1{(2\pi h)^n}\int_{\gamma (s,z)} e^{i(z-y)\eta /h +i\widetilde W(s,\eta)/h}v(y)dyd\eta ,
\ee
where it only remains to determine the contour $\gamma (s,z)$. We have,
\begin{lem}\sl
For any $z\in \co^n\cap\{ |\Im z| >\delta_0\}$ and $s\geq 0$, the application
$$
\Psi_{s,z} : (y,\eta )\mapsto \Phi_0(y) - \Im ((z-y)\eta +\widetilde W(s,\eta ))
$$
admits a saddle point at $U_s^{-1}(z, -\Im z)$ with critical value $\Phi_0(z)$.
\end{lem}
\begin{proof}
We compute,
\begin{eqnarray*}
&& \nabla_y\left( \Phi_0(y) - \Im ((z-y)\eta +\widetilde W(s,\eta ) )\right) = -\frac{i}2(\Im y + \eta);\\
&& \nabla_\y\left( \Phi_0(y) - \Im ((z-y)\eta +\widetilde W(s,\eta ))\right) =\frac{i}2 (z-y +\nabla_\y \widetilde W(s,\y)),
\end{eqnarray*}
so that any critical point must verify $\eta = -\Im y$ and $y = z +\nabla_\y \widetilde W(s,\y)$.
In particular, since $\eta$ is real, one must have $\Im y =\Im z$, and therefore,
$$
\eta = -\Im z\quad ; \quad y=z +\nabla_\y \widetilde W(s,-\Im z),
$$
that is, $(y,\eta )= U_s^{-1}(z, -\Im z)$.
Conversely, we  see that $U_s(z, -\Im z)$ is a critical point of $\Phi_0(y) - \Im ((z-y)\eta +\widetilde W(s,\eta ))$. Moreover, since $W$ is real on the real, we have  $\Im\nabla_\eta^2\widetilde W(s,\eta)= 0$ for $\eta=-\Im z$, and one easily deduces that the critical point is non degenerate for all $s\geq 0$. Since, for $s=0$, it is a saddle point, by continuity it remains a saddle point for all $s\geq 0$, and the critical value is $\Phi_0(z +\nabla_\y \widetilde W(s,-\Im z)) =\Phi_0(z)$.

\end{proof}

Thus, we could take for $\gamma (s, z)$ any good $2n$-contour for the map $(y,\eta )\mapsto \Phi_0(y) - \Im ((z-y)\eta +\widetilde W(s,\eta ))$, that is, any contour of real dimension $2n$, containing $U_s(z,-\Im z)$, and along which,
$$
\Psi_{s,z}(y,\eta )-\Phi_0(z)\leq -\delta |(y,\eta )-U_s(z,-\Im z)|^2,
$$
for some $\delta >0$.  However, $\delta$ may depend on $s$, and this is not enough to properly define the operator $G_0(s)$ when $s$ becomes large. 

\bigskip
Setting $A_s(z):= {\rm Hess}_\xi W(s, -\Im z)$, for $s\geq 0$ we instead consider the contour $\gamma (s,z)$ defined by,
\be
\label{contourG0}
\left\{
\begin{array}{l}
 \Re \eta = -\Im y ;\\
 \Im \eta = \la s\ra^{-2}\big[ \Re (z-y)+\partial_\zeta \widetilde W (s, -\Im z)+A_s(z)\Im (z-y)\big];\\
\la s\ra^{-1}| \Re (z-y)+\partial_\zeta \widetilde W (s, -\Im z)| + |\Im (z-y)|\leq r,
\end{array}
\right.
\ee
where $r>0$ is an arbitrarily small constant. Then, we claim,
\begin{lem}\sl 
\label{leboncontour}
If $r$ is chosen sufficiently small, then there exists a constant $\delta >0$ independent of $s$, such that, 
along $\gamma (s,z)$, one has,
$$
\Psi_{s,z}(y,\eta ) -\Phi_0(z) \leq -\delta \big[\la s\ra^{-2}| \Re (z-y)+\partial_\xi \widetilde W (s, -\Im z)|^2 + |\Im (z-y)|^2\big].
$$
\end{lem}
\begin{rem}
In particular, on the boundary of $\gamma (s,z)$ we obtain $\Psi_{s,z}(y,\eta ) -\Phi_0(z) \leq -r_1$ with $r_1>0$ independent of $s$. In that case, we will say that $\gamma (s,z)$ is a good contour for the phase $\Psi_{s,z}$ uniformly with respect to $s$.
\end{rem}
\begin{proof} 
At first, let us observe that, by Lemma \ref{estflow} and (\ref{lienWflot}), for any $\alpha\in \na^n$, we have,
\be
\label{estHessW}
\partial_\xi^\alpha W(s,\xi ) ={\cal O}(\la s\ra )
\ee
uniformly for $s={\cal O}(1/h)$.

\bigskip
We set,
\begin{eqnarray*}
&& a:= y-z-\nabla_\eta \widetilde W (s,-\Im z);\\
&& \alpha := \eta +\Im z.
\end{eqnarray*}
Since $W(s,\eta )$ is real on the real, performing a Taylor expansion and using  (\ref{estHessW}), we see that,
$$
\Im (\widetilde W(s,\eta ) - \alpha\cdot \nabla_\eta \widetilde W(s,-\Im z)) = A_s(z)\Re\alpha\cdot \Im\alpha
+{\cal O}(\la s\ra |\Im\alpha|\, |\alpha|^2).
$$
Then, by a straightforward computation, we find,
\begin{eqnarray*}
\Psi_{s,z}(y,\eta ) = \frac12 |\Im a|^2  +\Im a\cdot \Re\alpha +\Re a\cdot\Im \alpha -  A_s(z)\Re\alpha\cdot \Im\alpha\\
+{\cal O}(\la s\ra |\Im\alpha|\, |\alpha|^2).
\end{eqnarray*}
Now, by definition, along $\gamma (s,z)$, we have,
$$
\left\{
\begin{array}{l}
 \Re \alpha = -\Im a ;\\
 \Im \alpha =- \la s\ra^{-2}\big( \Re a+A_s(z)\Im a\big);\\
\la s\ra^{-1}| \Re a| + |\Im a|\leq r,
\end{array}
\right.
$$
and thus, there we obtain,
\begin{eqnarray*}
\Psi_{s,z}(y,\eta ) = -\frac12 |\Im a|^2   -\la s\ra^{-2}\big| \Re a+A_s(z)\Im a\big|^2\\
+{\cal O}(\la s\ra^{-1} \big|\Re a+A_s(z)\Im a\big| \, |\Im a|^2)\\
+{\cal O} (\la s\ra^{-5} \big|\Re a+A_s(z)\Im a\big|^3).
\end{eqnarray*}
Using again (\ref{estHessW}), this in particular gives,
\begin{eqnarray*}
&&\Psi_{s,z}(y,\eta ) \leq -\frac12 |\Im a|^2   -\la s\ra^{-2}\big| \Re a+A_s(z)\Im a\big|^2\\
&&\hskip 1cm +C\big(\la s\ra^{-1} |\Re a| +|\Im a|\big) \big( |\Im a|^2 
+\la s\ra^{-2} \big|\Re a+A_s(z)\Im a\big|^2\big),
\end{eqnarray*}
where $C>0$ is a uniform constant. Hence,
$$
\Psi_{s,z}(y,\eta ) \leq -(\frac12-Cr) |\Im a|^2   -(1-Cr)\la s\ra^{-2}\big| \Re a+A_s(z)\Im a\big|^2,
$$
and the result follows by taking $r < 1/(2C)$ and by observing that 
$$
|\Im a|^2   +\la s\ra^{-2}\big| \Re a+A_s(z)\Im a\big|^2  \geq \delta ( |\Im a|^2+\la s\ra^{-2}|\Re a|^2)
$$
if $\delta >0$ is small enough, independently of $s$.
\end{proof}
Taking the contour $\gamma (s,z)$ as in (\ref{contourG0}), and defining $G_0(s)$ by (\ref{defG0}) and $Z_s(z_0)$ by (\ref{defZs}), we have,
\begin{prop} \sl
\label{QEvol}
For $\varepsilon , s>0$ and $Z\in\co^n$, we set,
$$
\Omega_s(Z, \varepsilon ):=\{ z\in\co^n\,;\, \la s\ra^{-1}|\Re (z-Z)| + |\Im (z-Z)| <\varepsilon\}.
$$
Then, for any $z_0\in\co^n\cap\{|\Im z| >\delta_0\}$ and $\varepsilon_1>r$, there exists $\varepsilon_2>0$ such that, for any $s\geq 0$, the operator $G_0(s)$ maps $H_{\Phi_0}(\Omega_s(Z_s(z_0),\varepsilon_1))$ into $H_{\Phi_0}(\Omega_s(z_0,\varepsilon_2))$. Moreover, $G_0(s)$ verifies $G_0(0)\sim I$ and,
$$
ih\frac{\partial G_0}{\partial s} +  (\partial_s W)(s,hD_z) G_0(s) \sim 0,
$$
uniformly for $s\in [0, T/h]$ ($T>0$ arbitrary), in the sense that (possibly by shrinking $\varepsilon_2$) there exists a constant $C >0$ such that, for all $v\in H_{\Phi_0}(\Omega_s(Z_s(z_0),\varepsilon_1))$ and all $s\in [0, T/h]$, one has,
\begin{eqnarray*}
&& \big\Vert ih\frac{\partial G_0}{\partial s}v +  (\partial_s W)(s,hD_z) G_0(s)v \big\Vert_{L^2_{\Phi_0} (\Omega_s(z_0,\varepsilon_2))}\\
&& \hskip 6cm \leq Ce^{-1/Ch}\Vert v\Vert_{L^2_{\Phi_0} (\Omega_s(Z_s(z_0),\varepsilon_1))}.
\end{eqnarray*}
Here, we have used the notation $L^2_{\Phi_0} (\Omega ) =L^2(\Omega ; e^{-2\Phi_0(z)/h}d\Re zd\Im z)$.
\end{prop}

\begin{proof} Let us recall from \cite{Sj, MNS2} that, if $\Omega\subset \co^n$ is open, what we call $H_{\Phi_0}(\Omega)$ is the set of $h$-depending smooth functions on $\Omega$, that can be written on the form $F+f$, where $F$ is holomorpic on $\Omega$ and verifies $|F(z;h)|={\cal O}(e^{(\Phi_0(z)+\varepsilon)/h})$ for any $\varepsilon >0$, while $f$ is such that, for any $\alpha\in\na^{2n}$, there exists $\varepsilon_\alpha >0$ such that $|\partial_{(\Re z,\Im z)}^\alpha f(z;h)|={\cal O}(e^{(\Phi_0(z)-\varepsilon_\alpha )/h})$ uniformly for $z\in\Omega$ and $h>0$ small enough. Moreover, such two  functions are identified if their difference satisfies to similar estimates as the previous $f$, and then we also say that they are equivalent in $H_{\Phi_0}(\Omega)$, or that their difference is neglectible in $H_{\Phi_0}(\Omega)$.

\bigskip
In our case, we have $\Psi_{s,z}(y,\eta )-\Phi_0(z)\leq -\delta r^2/2$ on the boundary of $\gamma (s,z)$. Thus, by Cauchy theorem, for $z\in \Omega_s(z_0,\varepsilon_2)$  and $v\in H_{\Phi_0}(\Omega_s(Z_s(z_0),\varepsilon_1))$,  if we substitute $\gamma_0(s):=\gamma (s, z_0)$ to $\gamma (s,z)$ in the expression of $G_0(s)v$, we obtain a function that differs from $G_0(s)v$ by a neglectible function in $H_{\Phi_0}(\Omega_s(z_0),\varepsilon_2))$, where $\varepsilon_2>0$ must be sufficiently small in order to have,
$$
\sup_{s\geq 0}\,\,\sup_{z\in \Omega_s(z_+(x_0,\xi_0),\varepsilon_2)}\,\,\sup_{ (y,\eta)\in \partial\gamma_0(s)}(\Psi_{s,z}(y,\eta )-\Phi_0(z))<0.
$$
This is indeed possible, because on $\partial\gamma_0(s)$, one has,
$$
\la s\ra^{-1}|\Re (y-Z_s)| +|\Im (y-Z_s)| =r,
$$
and,
\begin{eqnarray*}
\Re (Z_s(z_0)-z-\partial_\xi \widetilde W(s,-\Im z) )&=& {\cal O}(|\Re (z_0-z)| )\\
&&  +{\cal O}( \la s\ra |\Im (z_0-z)|) ,
\end{eqnarray*}
while,
$$
\Im (Z_s(z_0)-z-\partial_\xi \widetilde W(s,-\Im z) )={\cal O}(|\Im (z_0-z)|).
$$
On the other hand, the function obtained by substituting $\gamma_0(s)$ to $\gamma (s,z)$ is clearly holomorphic in $\Omega_s(z_0,\varepsilon_2)$, and, by Lemma \ref{leboncontour}, it also verifies the estimates that makes it an element of $H_{\Phi_0}(\Omega_s(z_0,\varepsilon_2))$.

\bigskip
By construction, we also have $G_0(0)\sim I$, in the sense that there exists $C>0$ such that, for all $v\in  H_{\Phi_0}(|z-z_0|<\varepsilon_1)$, and possibly by shrinking $\varepsilon_2$, one has,
$$
\Vert G_0(0)v -v\Vert_{L^2_{\Phi_0}(|z-z_0|<\varepsilon_2)}\leq Ce^{-1/Ch}\Vert v\Vert_{L^2_{\Phi_0}(|z-z_0|<\varepsilon_1)}.
$$

Moreover, by definition we have,
\begin{eqnarray}
&&(\partial_sW)(s, hD_z)G_0(s) v(z;h)\nonumber\\
\label{dW0G0}
&&  = \frac1{(2\pi h)^{2n}}\int_{\Gamma_0(s,z)}e^{i(z-z')\zeta/h +i(z'-y)\eta /h +i\widetilde W(s,\eta)/h}\partial_\zeta W(s,\zeta )\\
&& \hskip 8cm\times v(y) dz'd\zeta dyd\eta,\nonumber
\end{eqnarray}
where $\Gamma_0(s,z)$ is the $4n$-contour defined by,
\be
\label{chcont1}
\Gamma_0(s,z)\, :\, \zeta = -\Im z +iR(\overline{z-z'})\quad ;\quad |z-z'|\leq r\quad ;\quad (y,\eta)\in\gamma(s,z'),
\ee
with an arbitrarily small constant $r>0$. It is easy to check that $\Gamma_0 (s,z)$ is a good contour  for the phase $(z',y,\zeta,\eta) \mapsto \Phi_0(y)-\Im ((z-z')\zeta +(z'-y)\eta  +\widetilde W(s,\eta))$ uniformly with respect to $s$ (the critical point is given by $z'=z$, $\zeta =\eta =-\Im z$, and $y=z +\nabla_\y \widetilde W(s,-\Im z) $). 

\bigskip
Another good contour is $\Gamma_1(s,z)$, given by,
\be
\label{chcont2}
\Gamma_1 (s,z)\, :\, (y,\eta)\in\gamma(s,z)\quad ;\quad z'=z-iR^{-1}(\overline{\zeta -\eta})\quad ;\quad |\zeta -\eta |< r',
\ee
 (with $r'>0$ small enough). 
 
 \bigskip
 Since the domain of integration is not necessarily a small neighborhood of the critical point, the fact that one can substitute $\Gamma_2(s,z)$ to $\Gamma_1(s,z)$ is not  automatic as in general Sj\"ostrand's theory, but this fact is proved in Appendix (see Lemma \ref{defcont1}).
We obtain,
\begin{eqnarray*}
&&(\partial_sW)(s, hD_z)G_0(s) v(z;h)\\
&& \hskip 2cm  = \frac1{(2\pi h)^{n}}\int_{\gamma(s,z)}e^{i(z-y)\eta /h +i\widetilde W(s,\eta)/h}W_1(s,z,\eta )v(y) dyd\eta,
\end{eqnarray*}
where $W_1$ is given by,
\begin{eqnarray*}
W_1(s,\eta )
=\frac1{(2\pi h)^n}\int_{{z'=iR^{-1}(\overline{\eta -\zeta})}\atop{|\eta -\zeta |<r'}}e^{iz'(\eta -\zeta) /h}(\partial_sW)(s, \zeta)dz'd\zeta.
\end{eqnarray*}
Then, standard arguments (such as the analytic stationary phase theorem, see \cite{Sj}) show that $W_1(s,\eta )\sim \partial_sW (s,\eta )$ in the space of analytic symbols  near $ \xi^+(x_0,\xi_0)$. 

Let us also observe that, by construction,   $(\partial_sW)(s, \zeta)$ is uniformly bounded together with all its derivatives for $s\in [0,T/h]$ and $|\zeta|\geq r_0>0$. As a consequence, the previous equivalence of symbols is indeed uniform with respect to $\in [0,T/h]$.

Therefore, setting,
$$
G_0'(s)v(z):=\frac1{(2\pi h)^{n}}\int_{\gamma(s,z)}e^{i(z-y)\eta /h +i\widetilde W(s,\eta)/h}(\partial_sW)(s, \eta)v(y) dyd\eta,
$$
we deduce the existence, for any given $\varepsilon_1>r$,  of $\varepsilon_2>0$ and $C>0$, such that,
\begin{eqnarray*}
&& \big\Vert G_0'(s)v - (\partial_s W)(s,hD_z) G_0(s)v \big\Vert_{L^2_{\Phi_0} (\Omega_s(z_+(x_0,\xi_0),\varepsilon_2))}\\
&& \hskip 6cm \leq Ce^{-1/Ch}\Vert v\Vert_{L^2_{\Phi_0} (\Omega_s(Z_s,\varepsilon_1))}.
\end{eqnarray*}
On the other hand, by differentiating the expression of $G_0(s)$ and using that $\gamma (s,z)$ is a uniformly good contour for the phase $\Psi_{s,z}$, we immediately obtain,
$$
ih\frac{\partial G_0}{\partial s} \sim -G_0'(s),
$$
and Proposition \ref{QEvol} is proved.
\end{proof}

In the same way, one can construct an operator,
$$
G_1(s)\, :\, H_{\Phi_0}(\Omega_s(z_0,\varepsilon_1))\rightarrow H_{\Phi_0}(\Omega_s(Z_s(z_0),\varepsilon_2)),
$$
(where, as before, $\varepsilon_2>0$ depends on $\varepsilon_1>r$, of the form,
\be
\label{defG1}
G_1(s)v(z) =\frac1{(2\pi h)^n}\int_{\gamma' (s,z)} e^{i(z-y)\eta /h -i\widetilde W(s,\eta)/h}v(y)dyd\eta
\ee
with $\gamma' (s,z)$ given by,
\be
\label{contourG1}
\left\{
\begin{array}{l}
 \Re \eta = -\Im y ;\\
 \Im \eta = \la s\ra^{-2}\big[ \Re (z-y)+x_1-\partial_\xi W (s, -\Im z)-A_s(z)\Im (z-y)\big)];\\
\la s\ra^{-1}| \Re (z-y)-x_1+\partial_\xi W (s, -\Im z)| + |\Im (z-y)|\leq r,
\end{array}
\right.
\ee
such that $G_1(0)\sim I$ and,
$$
ih\frac{\partial G_1}{\partial s} -  G_1(s)\, (\partial_s W)(s,hD_z)  \sim 0,
$$
in a sense analog to that of Proposition \ref{QEvol}. Then, we also have,
\begin{eqnarray*}
&& G_1(s)G_0(s)\sim I \mbox{ on } H_{\Phi_0}(\Omega_s(Z_s(z_0), \varepsilon));\\
&& G_0(s)G_1(s)\sim I \mbox{ on } H_{\Phi_0}(\Omega_s(z_0,\varepsilon)),
\end{eqnarray*}
(where, as before, one actually has to schrink $\varepsilon$ in the estimates) and, in the sequel, we rather write $G_0(s)^{-1}$ for $G_1(s)$.

\section{The conjugated evolution equation}\label{sec-conj}
In order to exploit the results of the previous section, we first show,
\begin{prop}\sl
\label{estevol}
 Let $T>0$, $\tilde H$ be the operator defined in (\ref{defHtilde1})-(\ref{quantizPDO}), and  $\tilde Q =h^2\tilde H$. Let also $Z=Z(s)\in \co^n$ be such that $\Im Z(s)\not=0$ and $\la s\ra^{-1}|\Re Z(s)| +|\Im Z(s)|+|\Im Z(s)|^{-1}={\cal O}(1)$ uniformly. Then, for any  $\varepsilon>0$ small enough, and by choosing $R$ sufficiently small in the definition of $\tilde H$ (that is, in (\ref{contourPDO})),  there exists $C>0$, such that, for any $s\in [0,T/h]$, one has,
$$
\big\Vert ih\frac{\partial \tilde v}{\partial s} - \widetilde Q \tilde v(s)\big\Vert_{L^2_{\Phi_0}(\Omega_s(Z(s),\varepsilon))}\leq Ce^{-1/Ch}.
$$
\end{prop} 

\bigskip
\begin{proof}
We first observe that  $\T \circ D_x = D_z \circ \T$. Moreover,  since $\Im Z(s) = {\cal O}(1)$, we 
have  that $\Im z$ remains uniformly bounded on $\Omega_s(Z(s),\varepsilon)$. Therefore, using Lemma \ref{derivHphi} in the appendix, we see that  $hD_z$ is uniformly bounded from $L^2_{\Phi_0}(\Omega_s(2\varepsilon))$ to $L^2_{\Phi_0}(\Omega_s(\varepsilon))$. 

In view of the form of $H$ and $\tilde H$, and by bringing all the derivatives to the left, we easily deduce that is enough to prove that, for any $a=a(x)$ holomorphic in $\Sigma_\nu$ and verifying $a(x) ={\cal O}(\la x\ra^k)$ for some fixed $k\geq 0$, and setting $\tilde a(z,\zeta ):= a(z+i\zeta )$, one has,
\be
\big\Vert \T (au) -  {\rm Op}_R(\tilde a)\T u\big\Vert_{L^2_{\Phi_0}(\Omega_s(Z(s),\varepsilon))}\leq Ce^{-1/Ch}\Vert u\Vert_{L^2(\re^n)},
\ee
with some constant $C>0$, and for all $s\in [0,T/h]$ and $u\in L^2(\re^n)$.

We compute,
\be
\label{opatildeTu}
{\rm Op}_R(\tilde a)\T u(z) =\int_{\re^n}e^{-(z-x)^2/2h}b(x,z;h) u(x)dx,
\ee
with,
$$
b(x,z;h):= \frac1{(2\pi h)^n} \int_{\gamma_R(z)}e^{(z-y)(i\zeta -x+(y+z)/2)/h }a\big( \frac{z+y}2 +i\zeta \big)dyd\xi.
$$
Now, the contour $\gamma_R(z)$ is not really good  for the phase $\Re \big( (z-y)(i\zeta -x+(y+z)/2)\big)$, because it does not contain the critical point  given by $y=z$ and $\zeta =i(z-x)$ (the critical value is 0), unless $x=\Re z$. However, for $R$ sufficiently large, along $\gamma_R(z)$ we have,
\begin{eqnarray}
&& \Re\big((z-y)(i\zeta -x+(y+z)/2)\big)\nonumber\\
&& \hskip 3cm  =-(R+\frac12 )|\Re (z-y)|^2-(R-\frac12)|\Im (z-y)|^2 \nonumber\\
&&\hskip 7cm   -(x-\Re z)\Re (z-y)\nonumber\\
\label{phasecommutT}
&& \hskip 3cm  \leq -\frac{R}2|z-y|^2+R^{-1/2}|x-\Re z|.
\end{eqnarray}
In particular, on the boundary of $\gamma_R$, we obtain,
$$
\Re\big((z-y)(i\zeta -x+(y+z)/2)\big)\leq -\frac12+R^{-1/2}|x-\Re z|.
$$
Then, we divide $\Omega_s(Z(s),\varepsilon)$ into,
$$
\Omega_s(Z(s),\varepsilon)=\bigcup_{j=1}^{N_s}\Omega_j,
$$
where $\Omega_j\subset \{ |z-z_j(s)| <\varepsilon\}$ for some $z_j(s)\in \Omega_s(Z(s),\varepsilon)$, and $N_s={\cal O}(\la s\ra^n)$.

We also choose a cut-off function $\chi_\varepsilon\in C_0^\infty (\{ |x|\leq 4\varepsilon\})$ verifying $\chi_\varepsilon (x)=1$ for $|x|\leq 3\varepsilon$, and, for $y$ such that ${\rm dist}(y ,\Omega_j)\leq R^{-1/2}$, we write,
\begin{eqnarray}
\label{decompTu}
\T u(y) &={}& \T (\chi (x-\Re z_j(s))u)(y) + \T ((1-\chi (x-\Re z_j(s)))u)(y) \nonumber\\
&=:&  v_j(y) +w_j(y).
\end{eqnarray}
Then, if we take $R^{-1/2}=\varepsilon$ and if  $x$ is in the support of $1-\chi (x-\Re z_j(s))$, we have $|x-\Re y|\geq |x-\Re z_j(s)| -|y-z_j(s)|\geq \varepsilon $. We deduce from this, 
\be
\label{estwj1}
e^{-\Phi_0(y)/h}|w_j(y)|\leq e^{-\varepsilon^2/2h}\Vert u\Vert_{L^2},
\ee
and thus, by the properties of continuity of ${\rm Op}_R(\tilde a)$ (see \cite{Sj, MNS2}),
\be
\label{estwj2}
\Vert {\rm Op}_R(\tilde a)w_j\Vert_{L^2_{\Phi_0}(\Omega_j)}\leq C\la s\ra^k e^{-\varepsilon^2/2h}\Vert u\Vert_{L^2}.
\ee
On the other hand, using (\ref{phasecommutT}), if $x$ in the support of $\chi (x-\Re z_j(s))$ and $z\in \Omega_j$, along $\gamma_R(z)$ we obtain,
$$
\Re\big((z-y)(i\zeta -x+(y+z)/2)\big)\leq -\frac{R}2|z-y|^2+4\varepsilon^2.
$$
In particular,  on the boundary of $\gamma_R(z)$ we have,
$$
\Re\big((z-y)(i\zeta -x+(y+z)/2)\big)\leq -\frac14,
$$
and by taking a slightly modified  contour, of the form,
$$
\zeta = -\Im z +iR(\overline{z-y}) -i(x-\Re z)\chi_0\big( R^{1/2}(y-z)\big)\quad ;\quad |y-z|\leq R^{-1/2},
$$
with $\chi_0(0)=1$, $\chi_0(r)=0$ if $|r|\geq 1/2$, we see that in this way (and possibly by taking $R$ larger), we obtain a good contour for the phase $\Re \big( (z-y)(i\zeta -x+(y+z)/2)\big)$. Indeed, along the new contour we have,
$$
\Re\big((z-y)(i\zeta -x+(y+z)/2)\big)\leq -\frac{R}2|z-y|^2+{\cal O}(\varepsilon R^{1/2}|z-y|^2),
$$
where, actually, $\varepsilon R^{1/2}=1$. As a consequence, and still for $x$ in the support of $\chi (x-\Re z_j(s))$, we can apply the analytic stationary phase theorem to compute the symbol $b$ (see \cite{Sj}, Theorem 2.1). In particular, setting, 
$$
y':= \frac{y+z}2 +i\zeta\quad ;\quad \xi =-i(y-z),
$$
the integral becomes,
$$
b(x,z;h)=\frac1{(2\pi h)^n}\int_{\gamma'(z)}e^{i(x-y')\xi /h}a(y')dy'd\xi,
$$
where $\gamma'(z)$ is a necessarily good contour, along which,
\begin{eqnarray*}
&& |x-y'| =|x+i\zeta -(y+z)/2| \leq |\zeta -i(z-x)| + |y-z|/2\\
&& \hskip 5cm \leq 2|x-\Re z| +3|y-z|/2\leq CR^{-1/2};\\
 &&\Re (i(x-y')\xi )\leq -R|x-y'|^2/C,
\end{eqnarray*}
for some constant $C>0$ independent of $R$ and $j$. Therefore,
 we are reduced to the same situation as  in \cite{Sj}, Example 2.6, and, for all $N\geq 1$, we obtain,
$$
b(x,z;h)=a(x) + R_N(x,z;h),
$$
with,
$$
|R_N(x,z;h)|\leq C' h^N(N+1)^nC^{-N}N!\sup_{|y-x|\leq CR^{-1/2}}|a(y)|.
$$
Taking $N=1/(C_1h)$ with $C_1\gg C$, and using the assumption on $a$, we finally obtain,
$$
|b(x,z;h)-a(x)|\leq \la s\ra^k e^{-\delta /h},
$$
uniformly for $x\in {\rm Supp}\hskip 1pt \chi (x-\Re z_j(s))$ and $z\in\Omega_j$, and 
where $\delta >0$ does not depend on $j$.

Coming back to the decomposition (\ref{decompTu}), setting $\chi_j(x):= \chi (x-\Re z_j(s))$, and using (\ref{opatildeTu}), this easily leads to,
$$
|  {\rm Op}_R(\tilde a)\T(1-\chi_j)u (z)-\T(1-\chi_j)au (z)|\leq \la s\ra^ke^{(\Phi_0(z)-\delta)/h}\Vert u\Vert_{L^2},
$$
for all $z\in\Omega_j)$, and thus,  using (\ref{estwj1})-(\ref{estwj2}),
$$
\big\Vert \T(au) -  {\rm Op}_R(\tilde a)\T u\big\Vert_{L^2_{\Phi_0}(\Omega_j)}={\cal O}(\la s\ra^ke^{-\delta'/h}\Vert u\Vert_{L^2(\re^n)}),
$$
with $\delta'=\min(\delta, \varepsilon^2/2)$. Summing over $j\in \{ 1,\cdots, N_s\}$, we finally obtain,
$$
\big\Vert \T(au) -  {\rm Op}_R(\tilde a)\T u\big\Vert_{L^2_{\Phi_0}(\Omega_s(Z(s),\varepsilon))}={\cal O}(\la s\ra^{k+n}e^{-\delta'/h}\Vert u\Vert_{L^2(\re^n)}),
$$
and since $s={\cal O}(h^{-1})$, the result follows.
\end{proof}

\bigskip
Setting,
$$
w(s):= G_0(s)\tilde v(s) = G_0(s)\T (e^{-ihsH}u_0),
$$
we deduce from Propositions \ref{estevol} and  \ref{QEvol} that, for any $z_0\in\co^n\cap\{|\Im z| >\delta'\}$ and for some constants $\varepsilon >0$ and $C>0$,  $w(s)$ verifies,
\be
\label{equationconj}
\big\Vert ih\frac{\partial w}{\partial s} - L(s) w(s)\big\Vert_{L^2_{\Phi_0}(\Omega_s(z_0,\varepsilon))}\leq Ce^{-1/Ch},
\ee
where we have set,
\be
\label{operateurconj}
L(s):= G_0(s)\tilde QG_0(s)^{-1} -(\partial_sW)(s, hD_z).
\ee
Then, as in the $C^\infty$ case, the following analogue of \cite{Na4} Lemma 3.1 is essential.
\begin{prop}\sl
\label{propQconj} Let $(x_0,\xi_0)\in T^*\re^{n}\backslash 0$ be forward non trapping, and set,
$$
z_+(x_0,\xi_0):= x_+(x_0,\xi_0) -i\xi_+(x_0,\xi_0).
$$
Then, for any $\varepsilon >0$ small enough, the operator $B(s):=G_0(s)\widetilde QG_0(s)^{-1} $ is an analytic pseudodifferential operator on $H_{\Phi_0}(\Omega_s(z_+(x_0,\xi_0),\varepsilon))$, with a symbol $b$ verifying,
\begin{eqnarray*}
b (s,z,\zeta;h) =(\tilde q_0+h\tilde q_1+h^2\tilde q_2) (z+\partial_\zeta \widetilde W(s,\zeta), \zeta ) +{\cal O}(h\la s\ra^{-1-\sigma})\\
  +{\cal O}(h^2\la s\ra^{-\sigma} +h^3\la s\ra^{1-\sigma}),
\end{eqnarray*}
uniformly for $s\geq 0$, $z\in\Omega_s(z_+(x_0,\xi_0),\varepsilon))$, $|\zeta +\Im z|$ small enough, and $h>0$ small enough.
\end{prop}
\begin{proof}
We first observe that, by (\ref{defG1}), we have,
$$
\tilde QG_0(s)^{-1}v(z)=\frac1{(2\pi h)^n}\int_{\gamma' (s,z)} e^{i(z-y)\eta /h -i\widetilde W(s,\eta)/h}\tilde q(z,\eta ;h)v(y)dyd\eta,
$$
where $\tilde q(z,\eta ;h):= e^{-iz\eta /h}\widetilde Q (e^{i(\cdot)\eta/h})$ is nothing but the symbol of $\widetilde Q$ for the standard quantization (see \cite{Sj}, Section 4), and, using the symbolic calculus (see, e.g., \cite{Ma2} Section 2) and Assumption A, one finds,
$$
\tilde q = \tilde q_0 + h\tilde q_1 + h^2 \tilde q_2 +hr_0(z+i\eta ,\eta) +h^2r_1(z+i\eta ,\eta)+h^3r_2(z+i\eta ,\eta),
$$
where the $\tilde q_j$'s ($0\leq j\leq 2$) are defined in (\ref{defqj}), and the $r_j(z,\zeta)$'s are holomorphic functions of $(z,\zeta )$ on $\Sigma_\nu\times \co^n$, and verify,
$$
r_j(z,\zeta ;h)={\cal O}(\la z\ra^{j-1-\sigma}), \quad (j=0,1,2),
$$
uniformly with respect to  $z\in\Sigma_\nu$ and $h>0$ small enough, and locally uniformly with respect to $\zeta\in\co^n$.

\bigskip
Then, we have,
$$
G_0(s)\tilde QG_0(s)^{-1}v(z)=\frac1{(2\pi h)^{2n}}\int_{\Gamma (s,z)} e^{i\phi /h}\tilde q(y,\eta ;h)v(x)dxd\eta dy d\zeta,
$$
where
$$
\phi =\phi (s,x,y,z,\eta,\zeta):= (z-y)\zeta +(y-x)\eta +\widetilde W(s,\zeta )-\widetilde W(s,\eta ),
$$
and $\Gamma (s,z)$ is the $4n$-contour given by,
$$
\Gamma (s,z)\, :\quad (y,\zeta )\in \gamma (s,z)\, ;\, (x,\eta )\in \gamma'(s,y).
$$
(The $2n$-contours $\gamma (s,z)$ and $\gamma' (s,y)$ are defined in (\ref{contourG0}) and (\ref{contourG1}), respectively.) In particular, $\Gamma (s,z)$ is automatically a good contour for the phase $\Phi_0(x)-\Im\phi$, uniformly with respect to $s\geq 0$.

\bigskip
Now, we write,
$$
\widetilde W(s,\zeta )-\widetilde W(s,\eta ) =(\zeta -\eta )\widetilde W_1(s,\zeta,\eta ),
$$
where $\widetilde W_1$ is well defined and holomorphic near $\zeta =\eta =\xi_+(x_0,\xi_0)$, and is equal  to $\partial_\zeta \widetilde W(s,\zeta )$ when  $\eta =\zeta$. Thus, we have,
$$
\Phi = (z-x)\eta + (z-y+\widetilde W_1(s,\zeta, \eta))(\zeta -\eta ).
$$
Setting $y'=y-\widetilde W_1(s,\zeta, \eta)$ (and dropping the prime), we obtain,
\begin{eqnarray}
G_0(s)\tilde QG_0(s)^{-1}v(z)=\frac1{(2\pi h)^{2n}}\int_{\Gamma' (s,z)} e^{i(y-x)\eta/h + i(z-y)\zeta  /h}\nonumber\\
\label{conjugQtilde}
\times \tilde q(y+\widetilde W_1,\eta ;h)v(x)dxd\eta dy d\zeta,
\end{eqnarray}
with $\Gamma' (s,z)$  given by,
\be
\label{contourGamma'}
\Gamma' (s,z)\, :\quad (y+\widetilde W_1(s,\zeta,\eta ),\zeta )\in \gamma (s,z)\, ;\, (x,\eta )\in \gamma'(s,y+\widetilde W_1).
\ee
Of course, $\Gamma' (s,z)$ is necessarily a good contour for the new phase $\Phi_0(x)-\Im ((y-x)\eta + (z-y)\zeta )$, uniformly with respect to $s$, and another such uniform good contour is given by,
\be
\label{contourGamma'1}
\Gamma_1' (s,z)\, :
\left\{
\begin{array}{l}
\eta = -\Im z +i(\overline{z-x});\\
 \zeta =\eta +i (\overline{z-y});\\
|z-x|<r\, ;\, |z-y|<r,
 \end{array}
 \right.
\ee
with $r>0$ small enough.

\bigskip
As before, the possibility of substitution of $\Gamma_1' (s,z)$ to $\Gamma' (s,z)$ is not completely obvious, but we prove it in the Appendix (see Lemma \ref{defcont2}). As a consequence (and up to exponentially small terms in $H_{\Phi_0}(\Omega_s(z_+(x_0,\xi_0),\varepsilon)))$, we obtain,
\be
\label{Qconj=OPD}
G_0(s)\tilde QG_0(s)^{-1}v(z)=\frac1{(2\pi h)^{n}}\int_{\gamma (z)} e^{i(z-x)\eta/h}b(s,z,\eta )v(x)dxd\eta,
\ee
with,
$$
b(s,z,\eta )= \frac1{(2\pi h)^{n}}\int_{\gamma'(z,\eta)} e^{i(z-y)(\zeta -\eta ) /h}\tilde q(y+\widetilde W_1(s,\zeta, \eta),\eta ;h)dyd\zeta
$$
and where the two $2n$-contours of integration are given by,
\begin{eqnarray*}
&& \gamma (z)\, :\left\{
\begin{array}{l}
\eta = -\Im z +i(\overline{z-x});\\
|z-x|<r,\\
 \end{array}
 \right.
\\
&& \gamma' (z,\eta )\, :\left\{
\begin{array}{l}
  \zeta =\eta +i (\overline{z-y});\\
 |z-y|<r.
 \end{array}
 \right.
\end{eqnarray*}
Here, we must observe that, for $|\eta +\Im z|+|\zeta -\eta|+ |y-z|$ small enough and $z\in \Omega_s(z_+(x_0,\xi_0),\varepsilon))$, one has,
$$
y+\widetilde W_1(s,\zeta, \eta)=z_+(x_0,\xi_0)+\partial_\zeta \widetilde W (s,\xi_+(x_0,\xi_0)) +\la s\ra \theta
$$
with $\theta\in\co^n$, $|\theta|$ arbitrarily small (uniformly with respect to $s$). Moreover, by Lemma \ref{estflow} and (\ref{lienWflot}), we also have,
$$
|\partial_\zeta \widetilde W (s,\xi_+(x_0,\xi_0))|\geq \frac{s}{C} -C,
$$
for some constant $C>0$. Therefore, since $\partial_\zeta \widetilde W (s,\xi_+(x_0,\xi_0))$ is real, we deduce that $y+\widetilde W_1(s,\zeta, \eta)$ remains in $\Sigma_\nu$, and thus $\tilde q(y+\widetilde W_1(s,\zeta, \eta),\eta ;h)$ is a well defined symbol in the region of integration. As a consequence, we can apply the analytic stationary phase theorem again (exactly as in \cite{Sj} Example 2.6), and we obtain,
\begin{eqnarray*}
b(s,z,\eta ;h)&\sim& \sum_{k\geq 0}\frac{h^k}{i^kk!}(\nabla_y\cdot\nabla_\zeta)^k\tilde q(y+\widetilde W_1(s,\zeta, \eta),\eta ;h)_{\left\vert_{{y=z}\atop{\zeta =\eta}}\right.}\\
&\sim& \sum_{j\geq 0}h^jb_j(s,z,\eta )
\end{eqnarray*}
 in the sense of analytic symbols on $\{ (z,\eta )\, ;\, z\in\Omega_s(z_+(x_0,\xi_0),\varepsilon)),\, |\eta +\Im z|<r\}$, and uniformly with respect to $s$, with,
\begin{eqnarray*}
&& b_0(s,z,\eta ) = \tilde q_0 (z+\partial_\zeta \widetilde W(s,\eta), \eta );\\
&& b_1(s, z, \eta ) = \tilde q_1 (z+\partial_\zeta \widetilde W(s,\eta), \eta )+{\cal O}(\la s\ra^{-1-\sigma});\\
&& b_2(s, z, \eta ) = \tilde q_2 (z+\partial_\zeta \widetilde W(s,\eta), \eta )+{\cal O}(\la s\ra^{-\sigma});\\
&& b_j(s,z,\eta ) ={\cal O}(\la s\ra^{1-\sigma}) \quad (j\geq 3),  
\end{eqnarray*}
uniformly for $s\geq 0$, $z\in\Omega_s(z_+(x_0,\xi_0),\varepsilon))$, $|\eta +\Im z|<r$.

\bigskip
In view of (\ref{Qconj=OPD}), this concludes the proof of Proposition \ref{propQconj}.
\end{proof}

\section{Proof of Theorem \ref{mainth}}\label{sec-proofmainth}

We start from (\ref{equationconj})-(\ref{operateurconj}) with $z_0=z_+(x_0,\xi_0)$, where $(x_0,\xi_0)$ is forward non trapping. Proposition \ref{propQconj} tells us that $L(s)$ is an analytic pseudodifferential operator on $H_{\Phi_0}(\Omega_s(z_+(x_0,\xi_0),\varepsilon))$, with symbol $\ell (s,z,\zeta;h)$ verifying,
\begin{eqnarray*}
\ell (s,z,\zeta;h) =(\tilde q_0+h\tilde q_1+h^2\tilde q_2) (z+\partial_\zeta \widetilde W(s,\zeta), \zeta ) -(\partial_sW)(s, \zeta)\\
 +{\cal O}(h\la s\ra^{-1-\sigma} +h^2\la s\ra^{-\sigma} +h^3\la s\ra^{1-\sigma}).
\end{eqnarray*}
In particular, uniformly for $s\in[0,T/h]$ with $T>0$ fixed,  we obtain,
$$
\ell (s,z,\zeta;h) =q (z+i\zeta +\partial_\zeta \widetilde W(s,\zeta), \zeta ;h ) -(\partial_sW)(s, \zeta)+{\cal O}(h\la s\ra^{-1-\sigma}),
$$
where $q=q_0+hq_1+h^2q_2$ is as in Section \ref{secPrelim}.

\bigskip
Then, denoting by $\kappa (x,\xi ):=(x-i\xi,\xi)$ the complex canonical map associated with $\T$, we observe that the Hamilton flow $\widetilde R_s$ of $\ell_0 (s,z,\zeta;h):= q (z+i\zeta +\partial_\zeta \widetilde W(s,\zeta), \zeta ;h ) -(\partial_sW)(s, \zeta)$ is given by,
$$
\widetilde R_s = \kappa \circ R_s\circ \kappa^{-1},
$$
with,
$$
R_s(x,\xi ;h):=(\tilde x(s,x,\xi ;h) -\partial_\zeta \widetilde W(s, \tilde\xi (s,x,\xi ;h)), \tilde\xi (s,x,\xi;h)),
$$
(where, as before, $(\tilde x(s,x,\xi ;h), \tilde \xi(s,x,\xi ;h))=\exp sH_q(x,\xi)$).

\bigskip
Then, by Lemma \ref{modclassevol}, we see that for any fixed $t>0$ and for any $(x,\xi)$ in a small enough complex neighborhood of $(x_0,\xi_0)$, one has
$$
\lim_{h\rightarrow 0_+} \widetilde R_{t/h} (x-i\xi,\xi;h)=(x_+(x,\xi)-i\xi_+(x,\xi), \xi_+(x,\xi)).
$$
Let us also observe that, by construction, we have,
$$
\ell_0 (s,z,\zeta;h)= q (z+i\zeta +\partial_\zeta \widetilde W(s,\zeta), \zeta ;h )-q (\partial_\zeta W(s,\zeta), \zeta ;h ),
$$
and thus, for $(z, \zeta)$ close enough to $(z_+(x_0,\xi_0),\xi_+(x_0,\xi_0))$, we obtain,
$$
\ell_0 (s,z,\zeta;h)={\cal O}(\la s\ra^{-\sigma-1}+h\la s\ra^{-\sigma} +h^2\la s\ra^{1-\sigma})={\cal O}(\la s\ra^{-\sigma-1}),
$$
uniformly for $h>0$ small enough and $s\in [0,T/h]$.

\bigskip
>From this point, the proof becomes very similar to that of the short range case \cite{MNS2}, and we only sketch it.

\bigskip
Setting $(\tilde z_s(z,\zeta ),\tilde\zeta_s(z,\zeta )):=\widetilde R_s(z,\zeta)$ and $z_1:=\tilde z_{s_1} (x_0-i\xi_0, \xi_0)$ with $s_1>0$ large enough, then for all $s\geq 0$, one first construct an analytic  Fourier integral operator,
$$
F(s) : H_{\Phi_0,z_1} \rightarrow H_{\Phi_0, \tilde z_s(z_1,-\Im z_1)},
$$
of the form,
$$
F(s)v (z) =\frac1{(2\pi h)^n}\int_{\gamma_s (z)} e^{i(\psi (s,z,\eta ) -y\eta )/h}v(y)dyd\eta,
$$
where  $\gamma_s (z)$ is a convenient $2n$-contour and $\psi $ is a holomorphic function of $(z,\eta)$ near $(\tilde z_s(z_1,-\Im z_1),-\Im z_1)$,  solution of the system (eikonal equation),
\begin{equation}
\label{eiko}
\left\{
\begin{array}{l}
\partial_s\psi + \ell_0 (s,z,\nabla_z\psi;h ) =0; \\
 \psi\left\vert_{s=0}\right. = z.\eta,
\end{array}\right.
\end{equation}
In particular, $\psi$ also depends on $h$ (but in a well-controled way), and it quantizes the canonical map $\widetilde R_s$, in the sense that one has,
$$
(z,\nabla_z\psi (s,z,\eta ) )= \widetilde R_s (\nabla_\eta\psi (s, z,\eta )).
$$
Moreover, $F(s)$ verifies,
$$
ih\partial_sF(s) -L(s)F(s) = hF_1(s),
$$
where $F_1(s)\, :\, H_{\Phi_0,z_1} \rightarrow H_{\Phi_0, \tilde z_s(z_1,-\Im z_1)}$ has the same form as $F(s)$, but with some symbol $f_1(s,z,\eta;h)$ that is ${\cal O}(\jap{s}^{-1-\sigma})$ uniformly for  $s\in [0,T/h]$, $h>0$ small enough. Moreover, using that, for any $t>0$,   $\tilde z_{t/h}(z,\zeta)$ tends to $z_\infty (z,\zeta):= z_+(z+i\zeta,\zeta)$ as $h\rightarrow 0_+$, one can prove, as in \cite{MNS2} Section 6, that, for any given $\varepsilon_1>0$, there exists $\varepsilon_2>0$ such that, for all $h>0$ small enough,  $F(t/h)$ sends $H_{\Phi_0}(|z-z_1|<\varepsilon_2)$ into $H_{\Phi_0}(|z-z_\infty (z_1, -\Im z_1)|<\varepsilon_1)$.

\bigskip
Similarly, still for $s\in [0, T/h]$, one can construct a Fourier integral operator,
$$
\widetilde F(s) : H_{\Phi_0, \tilde z_s(z_1,-\Im z_1)}   \rightarrow H_{\Phi_0,z_1}
$$ 
verifying,
\be
\label{eqFtilde}
ih\partial_s\widetilde F(s) +\widetilde F(s)L(s) =h\widetilde F_1(s),
\ee
where the Fourier integral operator  $\widetilde F_1(s)$ has the same phase as $\widetilde F(s)$ and a symbol ${\cal O}(\la s\ra^{-1-\sigma})$. Moreover, for any $t>0$ and $\varepsilon_1>0$, there exists $\varepsilon_2>0$ such that, for all $h>0$ small enough,  $\widetilde F(t/h)$ sends $H_{\Phi_0}(|z-z_\infty (z_1, -\Im z_1)|<\varepsilon_2)$ into $H_{\Phi_0}(|z-z_1|<\varepsilon_1)$.

\bigskip
Then, setting,
$$
\tilde w(s):= \widetilde F(s)w(s)=\widetilde F(s) G_0(s)\T (e^{-ihsH}u_0),
$$
we deduce from (\ref{equationconj})-(\ref{eqFtilde}) that $\tilde w$ verifies,
$$
ih\partial_s\tilde w(s) =h\widetilde F_1(s)w(s)=h\widetilde F_1(s)\widetilde A(s)F(s)\tilde w(s)\quad \mbox{ in } H_{\Phi_0,z_1},
$$
where $\widetilde A(s)$ is a parametrix of the elliptic analytic pseudodifferential operator $A(s):=F(s)\widetilde F(s)$ on $H_{\Phi_0, \tilde z_s (z_1, -\Im z_1)}$.

\bigskip
Observing that $\widetilde F_1(s)\widetilde A(s)F(s)$ is an analytic pseudodifferential operator on $H_{\Phi_0,z_1}$, with symbol uniformly ${\cal O}(\la s\ra^{-1-s})$ for $s\in[0, T/h]$, we easily deduce,
\be
\label{estderivenorm}
\big\vert \partial_s\Vert \tilde w(s)\Vert_{L^2_{\tilde\Phi_0}(|z-z_1|<\varepsilon_1)}^2\big\vert \leq C\la s\ra^{-1-\sigma}\Vert \tilde w(s)\Vert_{L^2_{\tilde\Phi_0}(|z-z_1|<2\varepsilon_1)}^2,
\ee
for some positive constants $\varepsilon_1$ small enough and $C$ large enough, and where $\tilde\Phi_0$ is such that $|\tilde\Phi_0 -\Phi_0|$ and $|\nabla_{(z,\overline z)}(\tilde \Phi_0 -\Phi_0 )|$ are small enough, $  \tilde \Phi_0\geq \Phi_0$ on $\{ |z-z_1|\leq \varepsilon_1\}$, 
$\tilde \Phi_0 = \Phi_0$ on $\{ |z-z_1|\leq \varepsilon_1/4\}$, and 
$\tilde \Phi_0>\Phi_0 +\delta_1$ on $\{ |z-z_1|\geq \varepsilon_1/2\}$ for some constant $\delta_1>0$.

\bigskip
Note that, in (\ref{estderivenorm}), the estimate is better than the analogous one obtained for the short range case in \cite{MNS2} (see formula (5.6) in \cite{MNS2}). This is due to the fact that, in (\ref{eiko}), we have left the whole symbol $\ell_0$ instead of taking just its principal part $q_0 (z+i\zeta +\partial_\zeta \widetilde W(s,\zeta), \zeta ;h )-q_0 (\partial_\zeta \widetilde W(s,\zeta), \zeta ;h )$. In \cite{MNS2}, this had some interest because the symbol obtained in this way did not depend on $h$, and the construction could be done for all $s\geq 0$ (without limitation of order $h^{-1}$). But here, in any case we have  an $h$-dependence inside $\widetilde W$.

\bigskip
Because of the choice of $\tilde\Phi_0$, and the fact that $\Vert e^{-ihsH}u_0\Vert_{L^2} =\Vert u_0\Vert_{L^2}$ does not depend on $s$, we deduce from (\ref{estderivenorm}),
$$
\big\vert \partial_s\Vert \tilde w(s)\Vert_{L^2_{\tilde\Phi_0}(|z-z_1|<\varepsilon_1)}^2\big\vert \leq C\la s\ra^{-1-\sigma}\Vert \tilde w(s)\Vert_{L^2_{\tilde\Phi_0}(|z-z_1|<\varepsilon_1)}^2+Ce^{-\delta_1/h},
$$
with some constant $\delta_1>0$, and thus, since $\int_0^\infty \la s\ra^{-1-\sigma}ds <\infty$, we obtain,
\begin{eqnarray}
\label{propagineq1}
\Vert \tilde w(s)\Vert_{L^2_{\tilde\Phi_0}(|z-z_1|<\varepsilon_1)}^2\leq C'\Vert \tilde w(0)\Vert_{L^2_{\tilde\Phi_0}(|z-z_1|<\varepsilon_1)}^2+C'e^{-\delta_1/h};\\
\label{propagineq2}
\Vert \tilde w(0)\Vert_{L^2_{\tilde\Phi_0}(|z-z_1|<\varepsilon_1)}^2\leq C'\Vert \tilde w(s)\Vert_{L^2_{\tilde\Phi_0}(|z-z_1|<\varepsilon_1)}^2+C'e^{-\delta_1/h},
\end{eqnarray}
with some new constant $C'>0$, and where the inequalities holds for all $s\in [0,T/h]$, $h>0$ small enough. 

\bigskip
Now, if  $(x_0,\xi_0)$ is not in $WF_a (u_0)$, then, by standard propagation of singularities (see, e.g., \cite{Ma2}),  we know that $R_{s_1}(x_0,\xi_0)\notin WF_a(u)$, and thus $\tilde w(0)$ is exponentially small in a neighborhood of $z_1:= \kappa \circ R_{s_1}(x_0,\xi_0)$. Then, (\ref{propagineq1}) tells us that, for all $s\in [0,T/h]$,  $\tilde w(s)$ is exponentially small in some fixed neighborhood of $z_1$. In particular, taking $s=t/h$ with $t>0$ fixed, we deduce that, for $h$ small enough, $G_0(t/h)\T (e^{-itH}u_0)$ is exponentially small near $z_\infty (z_1,-\Im z_1)=z_+(x_0,\xi_0)$. The converse can be seen in the same way by using (\ref{propagineq2}), and thus Theorem \ref{mainth} is proved.

\bigskip

\appendix
\vskip 1cm
\centerline{\bf APPENDIX}

\bigskip
\section{Deformation of non-local good contours}\label{appA}
\begin{lem}\sl
\label{defcont1}
Let $\Gamma_0(s,z)$ and $\Gamma_1 (s,z)$ be the two contours given in (\ref{chcont1})-(\ref{chcont2}) with $R\geq 1$ large enough.\\
 Then, there exists a deformation of contours $[0,1]\ni t\mapsto \widetilde\Gamma_t(s,z)$ with $ \widetilde\Gamma_0(s,z)\subset\Gamma_0(s,z)$ and $ \widetilde\Gamma_1(s,z)\subset \Gamma_1(s,z)$, such that, for all $t\in [0,1]$, $ \widetilde\Gamma_t(s,z)$ is a good contour for the phase $(z',y,\zeta,\eta) \mapsto \Phi_0(y)-\Im ((z-z')\zeta +(z'-y)\eta  +\widetilde W(s,\eta))$ uniformly with respect to $s$, and, for $j=0,1$,  one has,
 \begin{eqnarray*}
 \inf_{s\geq 0}\,\, \inf_{\Gamma_j(s,z)\backslash\widetilde\Gamma_j(s,z)}\big[\Phi_0(y)-\Phi_0(z)-\Im ((z-z')\zeta +(z'-y)\eta  +\widetilde W(s,\eta))\big]\\
 <0.
 \end{eqnarray*}
 Moreover, for any $\varepsilon_1>0$, if $z\in \Omega_s(z_+(x_0,\xi_0),\varepsilon_2)$ with $\varepsilon_2>0$ small enough, then, along $\Gamma_t(s,z)$, both $\eta$ and $\zeta$ remain in an arbitrary small neighborhood of $\xi_+(x_0,\xi_0)$, while $y$ remains in $\Omega_s (Z_s,\varepsilon_1)$. In particular, in the integral (\ref{dW0G0}) one can substitue $\Gamma_1(s,z)$ to $\Gamma_0 (s,z)$, up to an exponentially small term in $H_{\Phi_0}(\Omega_s(z_+(x_0,\xi_0),\varepsilon_2)))$.
\end{lem}
\begin{proof} First of all, using the fact that $\widetilde W$ is real on the real, by a Taylor expansion we obtain,
\begin{eqnarray*}
\Im \widetilde W (s,\eta )&=&\Im\eta\cdot \nabla_\zeta\widetilde W(s,\Re\eta) + {\cal O}(\la s\ra |\Im\eta|^3)\\
&=& \Im\eta\cdot \nabla_\zeta\widetilde W(s,-\Im z) + \Im\eta \cdot A_s(z)(\Re\eta +\Im z)\\
&&\hskip 2cm + {\cal O}(\la s\ra |\Im\eta|^3+\la s\ra |\Im\eta|\, |\Re\eta +\Im z|^2).
\end{eqnarray*}
Then, inserting this expression into the phase 
$$
\Phi_1:= \Phi_0(y)-\Phi_0(z)-\Im ((z-z')\zeta +(z'-y)\eta  +\widetilde W(s,\eta)),
$$
and setting
\begin{eqnarray}
\tilde\eta&=&  \Re\eta+\Im z+i\la s\ra\Im\eta;\nonumber\\
\tilde y&=& \la s\ra^{-1}\big(\Re(y-z)-\nabla_\zeta\widetilde W(s,-\Im z)\big) +i\Im (y-z);\nonumber\\
\label{chvar} {}\\
 \tilde \zeta&=& \zeta +\Im z;\nonumber\\
\tilde z&=&z'-z\nonumber,
\end{eqnarray}
we easily obtain,
\begin{eqnarray}
\Phi_1 &=& Q_s(\tilde\eta,  \tilde\zeta, \tilde y, \tilde z)+{\cal O}(\la s\ra^{-2} |\Im\tilde\eta|^3+ |\Im\tilde\eta|\, |\Re\tilde\eta |^2)\nonumber\\
\label{apphasequadrat}
&=& Q_s(\tilde\eta,  \tilde\zeta, \tilde y, \tilde z)+{\cal O}( |\tilde\eta|^3),
\end{eqnarray}
where $Q_s$ is a  uniformly non-degenerate real-quadratic form on $\co^{4n}$, with uniformly bounded coefficients,  such that, along the two contours $\Gamma_0(s,z)$ and $\Gamma_1(s,z)$, one has,
$$
Q_s(\tilde\eta,  \tilde\zeta, \tilde y, \tilde z)\leq   -\delta \big( |\tilde\eta|^2 + |\tilde\zeta|^2 \big),
$$
for some positive constant $\delta$. Actually, we find,
\begin{eqnarray*}
Q_s(\tilde\eta,  \tilde\zeta, \tilde y, \tilde z)&=&\frac12 (\Im \tilde y)^2 +\Im (\tilde z\,\tilde\zeta) +\Im (\tilde y\,\tilde\eta)-\Re\tilde\eta\,\Im\tilde z\\
&&-\frac1{\la s\ra}\Re\tilde z\,\Im\tilde\eta -\Im\tilde\eta\, \cdot \widetilde A_s(z)\Re\tilde\eta,
\end{eqnarray*}
where we have set,
$$
\widetilde A_s(z):=\la s\ra^{-1} A_s(z)={\cal O}(1).
$$

\bigskip
On the other hand, 
we see that the contour $\Gamma_0(s,z)$ is given by,
$$
\Gamma_0(s,z) :
\left\{
\begin{array}{l}
\tilde\eta = -i\overline{\tilde y}-i(\widetilde A_s(z)\Im\tilde y+a_s(z,\tilde z));\\
\tilde\zeta =-iR\overline{\tilde z}\,;\quad |\Re\tilde y+a_s(z,\tilde z)|+|\Im\tilde y|<r\, ;\, |\tilde z|<r,
\end{array}
\right.
$$
with,
\begin{eqnarray*}
a_s(z,\tilde z) &:=& \la s\ra^{-1}\big(\Re (z-z')+\partial_\zeta \widetilde W (s, -\Im z) -\partial_\zeta \widetilde W (s, -\Im z')\big)\\
 &{}=&-\la s\ra^{-1}\Re\tilde z +\widetilde A_s(z)\Im\tilde z +{\cal O}(|\Im\tilde z|^2),
\end{eqnarray*}
while $\Gamma_1(s,z)$ can be written as,
$$
\Gamma_1(s,z) :
\left\{
\begin{array}{l}
\tilde\eta = -i\overline{\tilde y}-i\widetilde A_s(z)\Im\tilde y;\\
\tilde\zeta =-iR\overline{\tilde z}+\Re\tilde\eta +i\la s\ra^{-1}\Im\tilde\eta;\\
|\Re\tilde y|+|\Im\tilde y|<r'\, ;\, |\tilde z|<r',
\end{array}
\right.
$$
that is,
$$
\Gamma_1(s,z) :
\left\{
\begin{array}{l}
\tilde\eta = -i\overline{\tilde y}-i\widetilde A_s(z)\Im\tilde y;\\
\tilde\zeta =-iR\overline{\tilde z}-\Im\tilde y-i\la s\ra^{-1}(\Re\tilde y+\widetilde A_s(z)\Im\tilde y);\\
|\Re\tilde y|+|\Im\tilde y|<r'\, ;\, |\tilde z|<r'.
\end{array}
\right.
$$
Then, for $t\in [0,1]$, we set,
$$
\widetilde\Gamma_t(s,z) :
\left\{
\begin{array}{l}
\tilde\eta = -i\overline{\tilde y}-i(\widetilde A_s(z)\Im\tilde y + (1-t)a_s(z,\tilde z));\\
\tilde\zeta =-iR\overline{\tilde z}-t\Im\tilde y-it\la s\ra^{-1}(\Re\tilde y+\widetilde A_s(z)\Im\tilde y);\\
|\tilde y|+ |\tilde z|<r'',
\end{array}
\right.
$$
where $r''>0$ is taken sufficiently small in order to have,
\begin{eqnarray*}
\{ |\tilde y|+ |\tilde z|<r''\} \subset \big(\{ |\Re\tilde y+a_s(z,\tilde z)|+|\Im\tilde y|<r\, ;\, |\tilde z|<r,\}\\
\cap \{|\Re\tilde y|+|\Im\tilde y|<r'\, ;\, |\tilde z|<r'\}\big).
\end{eqnarray*}
A straightforward computations shows that, along $\widetilde\Gamma_t(s,z)$, one has,
\begin{eqnarray*}
Q_s(\tilde\eta,  \tilde\zeta, \tilde y, \tilde z)\leq -\frac12(\Im \tilde y)^2 -(\Re \tilde y +\widetilde A_s(z)\Im\tilde y)^2 -R|\tilde z|^2\\
 +C|z|\, (|z|+|\Im\tilde y| +| \Re \tilde y +\widetilde A_s(z)\Im\tilde y|),
\end{eqnarray*}
where the constant $C$ does not depend on the choice of $R$. As a consequence, by choosing $R$ sufficiently large, along $\widetilde\Gamma_t(s,z)$ we obtain,
$$
Q_s(\tilde\eta,  \tilde\zeta, \tilde y, \tilde z)\leq -\delta\big( |\tilde y|^2 +|\tilde z|^2\big),
$$
with some constant $\delta >0$. Since $|\tilde y| +|\tilde z|\sim |\tilde \eta|+|\tilde \zeta|$ on $\widetilde\Gamma_t(s,z)$ (in the sense that both quotients of these quantities are uniformly bounded), in view of 
(\ref{apphasequadrat})
this means that $\widetilde\Gamma_t(s,z)$ is a good contour for the phase $\Phi_1$.

\bigskip
Finally, along on $\widetilde\Gamma_t(s,z)$, all vectors $\tilde\eta,  \tilde\zeta, \tilde y, \tilde z$ remain in an arbitrarily small neighborhood of 0. As a consequence, we see on (\ref{chvar}) that both $\eta$ and $\zeta$ remain in an arbitrary small neighborhood of $\xi_+(x_0,\xi_0)$, while $y$ remains in $\Omega_s (Z_s,\varepsilon_1)$ ($\varepsilon_1>0$ arbitrary)  if $z\in \Omega_s(z_+(x_0,\xi_0),\varepsilon_2)$ with $\varepsilon_2>0$ is small enough.
\end{proof}

\begin{lem}\sl
\label{defcont2}
Let $\Gamma' (s,z)$ and $\Gamma_1' (s,z)$ be the two contours given in (\ref{contourGamma'})-(\ref{contourGamma'1}).\\
 Then, there exists a deformation of contours $[0,1]\ni t\mapsto \Gamma_t(s,z)$ with $\Gamma_0(s,z)\subset\Gamma'(s,z)$ and $\Gamma_1(s,z)\subset \Gamma_1'(s,z)$, such that, for all $t\in [0,1]$, $\Gamma_t(s,z)$ is a good contour for the phase  $(x,y,\zeta,\eta)\mapsto \Phi_0(x)-\Im ((y-x)\eta + (z-y)\zeta )$ uniformly with respect to $s$, and one has,
 \begin{eqnarray*}
&&  \inf_{s\geq 0}\,\, \inf_{\Gamma'(s,z)\backslash\Gamma_0(s,z)}\big[\Phi_0(x)-\Phi_0(z)-\Im ((y-x)\eta + (z-y)\zeta )\big]<0;\\
&&  \inf_{s\geq 0}\,\, \inf_{\Gamma_1'(s,z)\backslash\Gamma_1(s,z)}\big[\Phi_0(x)-\Phi_0(z)-\Im ((y-x)\eta + (z-y)\zeta )\big]<0
 \end{eqnarray*}
 Moreover, if $z\in \Omega_s(z_+(x_0,\xi_0),\varepsilon)$ with $\varepsilon>0$ small enough, then, along $\Gamma_t(s,z)$, the quantity $y+\widetilde W_1(s,\zeta, \eta)$ remains inside $\Sigma_\nu$. In particular, in the integral (\ref{conjugQtilde}) one can substitue $\Gamma_1' (s,z)$ to $\Gamma' (s,z)$, up to an exponentially small term in $H_{\Phi_0}(\Omega_s(z_+(x_0,\xi_0),\varepsilon)))$.
\end{lem}
\begin{proof} 
We first observe that, along $\Gamma' (s,z)$, we have,
$$
|\Re \eta +\Im z| +|\Re\zeta+\Im z |\leq 2r\quad ;\quad |\Im \eta| +|\Im\zeta |\leq C_0\la s\ra^{-1}r,
$$
with a constant $C_0>0$ independent of the choice of $r$. Moreover, writing,
$$
\widetilde W_1(s,\zeta,\eta )=\int_0^1 \partial_\zeta \widetilde W(t\zeta + (1-t)\eta )dt,
$$
and using the fact that $\widetilde W$ is real on the real, we easily deduce,
\be
\label{ImW1}
\Im \widetilde W_1(s,\zeta ,\eta)=\frac12 A_s(z)(\beta_1+\beta_2 )+{\cal O}(|\alpha|\, |\beta| +\la s\ra^{-2}|\beta|^3),
\ee
where we have set,
\begin{eqnarray*}
&&\alpha =(\alpha_1,\alpha_2):= (\Re \zeta +\Im z,\Re\eta +\Im z);\\
&& \beta = (\beta_1,\beta_2) := (\la s\ra\Im\zeta ,\la s\ra\Im \eta ).
\end{eqnarray*}
Similarly, we have,
\be
\label{ReW1}
\Re \widetilde W_1(s,\zeta ,\eta)=\partial_\zeta \widetilde W(s,-\Im z) +\frac12 A_s(z)(\alpha_1+\alpha_2 )+{\cal O}(\la s\ra|\alpha|^2+\la s\ra^{-1}|\beta|^2).
\ee
Substituting these expressions into the equations defining  $\Gamma' (s,z)$, and setting $\widetilde A_s=\la s\ra^{-1}A_s(z)$, we find,
\begin{eqnarray*}
&& \alpha_1 +\frac12 \widetilde A_s(\beta_1+\beta_2)= \Im (z-y)+{\cal O}(|\alpha|\, |\beta| +|\beta|^3);\\
&& \beta_1 + \frac12 \widetilde A_s(\alpha_1+\alpha_2)+\frac12  \widetilde A_s^2(\beta_1+\beta_2 )\\
 &&\hskip 3cm=\la s\ra^{-1}\Re (z-y)+ \widetilde A_s\,\Im (z-y)+{\cal O}(|\alpha|^2+|\beta|^2);\\
 &&\alpha_2 = \Im (z-x);\\
&&\beta_2 - \frac12  \widetilde A_s(\alpha_1+\alpha_2 )+\frac12  \widetilde  A_s^2(\beta_1+\beta_2 )\\
&&\hskip 3cm= \la s\ra^{-1}\Re (y-x)- \widetilde A_s\,\Im (y-x)+{\cal O}(|\alpha|^2+|\beta|^2),
\end{eqnarray*}
where similar estimates for the derivatives  can also be obtained.

\bigskip
In particular, since $|\alpha| +|\beta|\leq (2+C_0)r \ll 1$, $ \widetilde A_s={\cal O}(1)$, and $ \widetilde A_s^2\geq 0$, we easily see that  the implicit function theorem can be applied to the variables $(\alpha, \beta)$, and  permits us to re-write  the equations defining $\Gamma' (s,z)$ as,
$$
\zeta = F_0(s, z,x,y) \quad  ;\quad \eta = G_0(s,z,x,y),
$$
where $F_0$ and $G_0$ depend smoothly  on $(z,x,y)$ (actually, analytically on $\Re(z,x,y)$ and $\Im (z,x,y)$) in the domain of integration, and verify there,
\begin{eqnarray}
&&\Phi_0(x)-\Phi_0(z)-\Im ((y-x)G_0 + (z-y)F_0 )\nonumber\\
\label{boncontN0}
&& \hskip 2cm \leq -\delta_1\la s\ra^{-2}\big[ |\Re (z-y)|^2+|\Re (z-x)|^2\big]\\
&& \hskip 5cm  -\delta_1 \big[|\Im (z-y)|^2+|\Im (z-x)|^2\big],\nonumber
\end{eqnarray}
for some constant $\delta_1 >0$.
(Here, we have also used the fact that, on $\Gamma' (s,z)$, the size of $(\alpha, \beta)$ is of the same order as that of $(\Im (z-y), \Im (z-x), \la s\ra^{-1}\Re (z-y), \la s\ra^{-1}\Re (z-x))$.)

Moreover, if $r_0>0$ is a small enough constant, the set,
\begin{eqnarray*}
\Gamma_0(s,z):=\big\{ (x,z,\zeta,\eta)\, ;\, \zeta = F_0(s, z,x,y),\,  \eta = G_0(s,z,x,y),\, \\
\la s\ra^{-2}|\Re (z-y)|^2+|\Im (z-y)|^2<r_0^2,\\
 \la s\ra^{-2}|\Re (z-x)|^2+|\Im (z-x)|^2<r_0^2\big\}
\end{eqnarray*}
is included in $\Gamma' (s,z)$ for all $s\geq 0$, and it verifies,
 $$
 \inf_{s\geq 0}\,\, \inf_{\Gamma'(s,z)\backslash\Gamma_0(s,z)}\big[\Phi_0(x)-\Im ((y-x)\eta + (z-y)\zeta )\big]<0.
 $$

\bigskip
On the other hand, the contour $\Gamma_1'(s,z)$ can obviously be written as,
$$
\zeta = F_1(s, z,x,y)\, ;\,  \eta = G_1(s,z,x,y)\, ;\, |z-x|<r_1\, ;\, |z-y|<r,
$$
with,
\begin{eqnarray}
\label{boncontN1}
&&\Phi_0(x)-\Phi_0(z)-\Im ((y-x)G_1 + (z-y)F_1 )\\
&& \hskip 6cm  \leq -\delta_2 \big[|z-y|^2+|z-x|^2\big],\nonumber
\end{eqnarray}
for some constant $\delta_2 >0$, and, possibly by shrinking  $r_0$, we can can also assume that $r_0\leq r$. 

\bigskip
Then, for $t\in[0,1]$, we set,
\begin{eqnarray*}
&& \rho_t (s):=\sqrt{ (1-t)\la s\ra^{-2} + t};\\
&& B_t(s,z):=\{x\in\co^n\,;\, \rho_t(s)^2|\Re (z-x)|^2 +|\Im (z-x)|^2<r_0^2\},
\end{eqnarray*}
and, for $j=0,1$, we define the vector-valued functions $x_t^j(s,z,x,)$ and $y_t^j(s,z,y)$ by the formulas,
\begin{eqnarray*}
\Re (x_t^j-z)=\la s\ra^{1-j}\rho_t(s)\Re (x-z)&;& \Im x_t^j=\Im x;\\
\Re (y_t^j-z)=\la s\ra^{1-j}\rho_t(s)\Re (y-z)&;& \Im y_t^j=\Im y.
\end{eqnarray*}
In particular, we see that if $x,y\in B_t(s,z)$, then $x_t^j,y_t^j \in B_j(t,z)$, and moreover $x_j^j=x$, $y_j^j=y$.
Therefore, we can consider the contour $\Gamma_t(s,z)$ given by,
$$
\Gamma_t(s,z)\,:\, \zeta = F_t(s, z,x,y)\, ;\,  \eta = G_t(s,z,x,y)\, ;\, x,y\in B_t(s,z),
$$
where $F_t$ and $G_t$ are defined by,
\begin{eqnarray*}
 \Re F_t(s,z,x,y)&=&(1-t)\Re F_0(s,z,x_t^0,y_t^0)+t\Re F_1(s,z,x_t^1,y_t^1);\\
 \Im F_t(s,z,x,y)&=&(1-t)\la s\ra\rho_t(s)\Im F_0(s,z,x_t^0,y_t^0)+t\Im F_1(s,z,x_t^1,y_t^1);\\
 \Re G_t(s,z,x,y)&=&(1-t) \Re G_0(s,z,x_t^0,y_t^0)+t \Re G_1(s,z,x_t^1,y_t^1);\\
 \Im G_t(s,z,x,y)&=&(1-t)\la s\ra\rho_t(s)\Im G_0(s,z,x_t^0,y_t^0)+t\Im G_1(s,z,x_t^1,y_t^1).
\end{eqnarray*}
(Note that the notations remain consistent when $t=0$ or $t=1$.) Then,  $t\mapsto \Gamma_t(s,z)$ is a continuous deformation between $\Gamma_0(s,z)$ and $\Gamma_1(s,z)$, and, along $\Gamma_t(s,z)$, a straightforward computation gives,
\begin{eqnarray*}
&&\Im ((y-x)\eta + (z-y)\zeta )\\
&& \hskip 2cm = (1-t)\Im ((y_t^0-x_t^0)G_0(x_t^0,y_t^0)+(z-y_t^0)F_0(x_t^0,y_t^0))\\
&& \hskip 3cm +t\Im ((y_t^1-x_t^1)G_1(x_t^1,y_t^1)+(z-y_t^1)F_1(x_t^1,y_t^1)),
\end{eqnarray*}
and thus, using (\ref{boncontN0}) and (\ref{boncontN1}), and the fact that $\Phi_0(x)=\Phi_0(x_t^0)=\Phi_0(x_t^1)$, we obtain,
\begin{eqnarray*}
&&\Phi_0(x) -\Phi_0(z) -\Im ((y-x)\eta + (z-y)\zeta )\\
&& \hskip 1.5cm \leq -\delta_1(1-t)\rho_t^2(s)\big[ |\Re (z-y)|^2+|\Re (z-x)|^2\big]\\
&& \hskip 4cm  -\delta_1 (1-t)\big[|\Im (z-y)|^2+|\Im (z-x)|^2\big]\\
&& \hskip 2cm  -\delta_2t\rho_t^2(s)\big[ |\Re (z-y)|^2+|\Re (z-x)|^2\big]\\
&& \hskip 4cm  -\delta_2 t\big[|\Im (z-y)|^2+|\Im (z-x)|^2\big].
\end{eqnarray*}
In particular, setting $\delta=\min (\delta_1,\delta_2)$, this gives,
\begin{eqnarray*}
&&\Phi_0(x) -\Phi_0(z) -\Im ((y-x)\eta + (z-y)\zeta )\\
&& \hskip 1.5cm \leq -\delta \rho_t^2(s)\big[ |\Re (z-y)|^2+|\Re (z-x)|^2\big]\\
&& \hskip 4cm  -\delta \big[|\Im (z-y)|^2+|\Im (z-x)|^2\big],
\end{eqnarray*}
and we also obtain that $\Phi_0(x) -\Phi_0(z) -\Im ((y-x)\eta + (z-y)\zeta )\leq -\delta r_0^2$ on the boundary of $\Gamma_t(s,z)$. As a consequence, $\Gamma_t(s,z)$ is a good contour for the phase $\Phi_0(x)-\Im ((y-x)\eta + (z-y)\zeta )$ uniformly with respect to $s$, and it remains to prove that,  if $r_0$ has been chosen small enough, and if $z\in \Omega_s(z_+(x_0,\xi_0),\varepsilon)$ with $\varepsilon>0$ small enough, then, for all $t\in[0,1]$, along $\Gamma_t(s,z)$ one has $y+\widetilde W_1(s,\zeta, \eta)\in \Sigma_\nu$.

\bigskip
We first observe  that, by definition, the functions $F_j$ and $G_j$ ($j=0,1$) verify,
\begin{eqnarray*}
|\Re F_j+\Im z| +|\Re G_j+\Im z| \leq C\rho_j(s)\big[ |\Re (z-y)|+|\Re (z-x)| \big] \\
+|\Im (z-y)|+|\Im (z-x)|,
\end{eqnarray*}
for some uniform constant $C>0$. Then, for any $t\in[0,1]$, we deduce,
\begin{eqnarray*}
|\Re F_t+\Im z| +|\Re G_t+\Im z| \leq C\rho_t(s)\big[ |\Re (z-y)|+|\Re (z-x)| \big] \\
+C\big[|\Im (z-y)|+|\Im (z-x)|\big],
\end{eqnarray*}
and thus, on $\Gamma_t(s,z)$,
$$
|\Re \zeta +\Im z| + |\Re \eta +\Im z| \leq 4Cr_0.
$$
Similarly, we find,
\begin{eqnarray*}
|\Im F_t| +|\Im G_t| \leq C((1-t)\rho_t(s) +t)\rho_t(s)\big[ |\Re (z-y)|+|\Re (z-x)| \big] \\
+C((1-t)\rho_t(s) +t)\big[|\Im (z-y)|+|\Im (z-x)|\big],
\end{eqnarray*}
and thus, on $\Gamma_t(s,z)$,
$$
|\Im \zeta | + |\Im \eta | \leq 4Cr_0.
$$
Therefore, using (\ref{ImW1})-(\ref{ReW1}), we deduce that, along $\Gamma_t(s,z)$, we have,
$$
y+\widetilde W_1(s,\zeta, \eta) =z + \partial_\zeta \widetilde W(s,-\Im z) + Y,
$$
where $z, Y\in \co^n$ verifies,
$$
|Y|\leq C_1 r_0\la s\ra,\quad |z-z_+(x_0,\xi_0)|\leq C_1\varepsilon\la s\ra,
$$
with a constant $C_1>0$ independent of $r_0, \varepsilon$. In particular, using (\ref{lienWflot}) and Lemma \ref{estflow}, we conclude that $y+\widetilde W_1(s,\zeta, \eta)\in \Sigma_\nu$ as long as $r_0$ and $\varepsilon$ are taken sufficiently small.
\end{proof}

\section{Derivatives on non-local $H_{\Phi_0}$-spaces}\label{appB}

With ${\rm Op}_R$ defined as in (\ref{quantizPDO}) and $Z(s)$ as in Proposition \ref{estevol}, we have,
\begin{lem}\sl
\label{derivHphi}
For all $\alpha\in\na^n$, $\varepsilon>0$, and $R\geq 1$ large enough, there exists $C>0$, such that,
$$
\Vert (hD_z)^{\alpha}v - {\rm Op}_R(\zeta^\alpha )v\Vert_{L^2_{\Phi_0}(\Omega_s(Z(s), \varepsilon))}\leq C\la s\ra^{n} e^{-1/Ch}\Vert v\Vert_{L^2_{\Phi_0}(\Omega_s(Z(s), \varepsilon+4R^{-1/2}))},
$$
uniformly with respect to $s\geq 0$, $v\in H_{\Phi_0}(\Omega_s(Z(s), \varepsilon_1))$,  and $h>0$ small enough.
\end{lem}
\begin{proof} When $\alpha=0$, if we parametrize $\gamma_R(z)$ by $y\in\co^n$, we find,
$$
 {\rm Op}_R(1)v(z;h) =\left( \frac{R}{\pi h}\right)^n\int_{|y-z|<R^{-1/2}} e^{-R|y-z|^2 /h}w_z( y)d\Re y\,d\Im y,
$$
with
\be
\label{SP1}
w_z( y):= e^{-i (z-y)\Im z/h}v(y).
\ee
Setting
$$
x:=R^{1/2}\left(  \Re (y-z),\Im (y-z)\right),
$$
and $\lambda:=2/h$,
the previous integral becomes,
$$
(\pi h)^{-n}\int_B e^{-\lambda x^2/2}\tilde w_z(x)dx,
$$
with $B:=\{ x:=(x_1,x_2)\in\re^{2n}\, ;\,  |x|<1\}$, and
$$
\tilde w_z(x):=w_z ( z+R^{-1/2}(x_1+ix_2)).
$$
Then, we can apply the analytic stationary phase theorem, as stated in \cite{Sj} Theorem 2.1, and, observing that $\Delta \tilde w_z (x)\equiv 0$, we obtain,
$$
 {\rm Op}_R(1)v(z;h) =v(z;h) + R_N(h),
$$
where, for all $N\geq 1$ and $h>0$, $R_N(h)$ verifies,
$$
|R_N(h)|\leq C_n 2^{-n}h^{n+N}(N+1)^n N! \,  \sup_{\tilde x\in\widetilde B} |\tilde w_z(\tilde x)|,
$$
for some constant $C_n>0$, and with $\widetilde B:=\{ \tilde x = \mu x\,;\, x\in B\,, \, \mu\in\co\,,\, |\mu|<1\}\subset \{ \tilde x=(\tilde x_1,\tilde x_2)\in\co^{2n}\, ;\, |\tilde x_1+i\tilde x_2| <2\}$.  Then, taking $N=[1/C_1h]$ with $C_1>1$, we easily deduce,
\be
\label{SP2}
| {\rm Op}_R(1)v(z;h) -v(z;h) |\leq Ce^{-1/Ch} \sup_{\tilde x\in\widetilde B} |\tilde w_z(\tilde x)|,
\ee
for some constant $C>0$ independent of $R$. Now, by Cauchy estimates, for any $\tilde x\in\co^{n}$ with $|\tilde x|<2$, we see that (for some other constant $C>0$ independent of $R$),
\be
\label{SP3}
|v(z+R^{-1/2}\tilde x)|\leq CR^{n/4}e^{(CR^{-1/2}+\Phi_0(z+R^{-1/2}\tilde x))/h}\Vert v\Vert_{L^2_{\Phi_0}(B_z(4R^{-1/2}))},
\ee
where $B_z(4R^{-1/2}):=\{ z'\in\co^n\,;\, |z'-z| < 4R^{-1/2}\}$.
On the other hand, we have,
\be
\label{SP4}
|-\Phi_0(z)+iR^{-1/2}\tilde x\, \Im z +\Phi_0(z+R^{-1/2}\tilde x)|=R^{-1}(\Im \tilde x)^2\leq 4R^{-1}.
\ee
We deduce from (\ref{SP1})-(\ref{SP4}),
$$
| {\rm Op}_R(1)v(z;h) -v(z;h) |\leq CR^{n/4}e^{\Phi_0(z)/h+ CR^{-1/2}-1/Ch} \Vert v\Vert_{L^2_{\Phi_0}(B_z(4R^{-1/2}))},
$$
for some (new) constant $C>0$ independent of $R$. Therefore, taking $R$ sufficiently large, for all $z$ in $\Omega_s(Z(s), \varepsilon)$, we obtain,
$$
e^{-\Phi_0(z)/h}| {\rm Op}_R(1)v(z;h) -v(z;h) |\leq CR^{n/4}e^{-1/2Ch} \Vert v\Vert_{L^2_{\Phi_0}(\Omega_s(\varepsilon +4R^{-1/2}))}.
$$
Taking the square and integrating with respect to $z$ on $\Omega_s(Z(s), \varepsilon)$, the result for $\alpha =0$ follows. Then, the general result for any $\alpha\in\na^n$ follows, too, by observing that ${\rm Op}_R(\zeta^\alpha )v=(hD_z)^\alpha {\rm Op}_R(1 )v+{\cal O}(e^{-1/2h}\sup_{|y-z|\leq R^{-1/2}} |v(y)|$.
\end{proof}

{}

\end{document}